\documentclass[12pt]{amsart}

\usepackage{amsmath, amssymb, graphics}
\usepackage{graphicx}

\newcommand{\mathsym}[1]{{}}
\newcommand{\unicode}[1]{{}}
\newtheorem{thm}{Theorem}[section]
\newtheorem{cor}[thm]{Corollary}
\newtheorem{lem}[thm]{Lemma}
\newtheorem{prop}[thm]{Proposition}
\theoremstyle{definition}

\newtheorem{defn}[thm]{Definition}

\newtheorem*{defn*}{Definition}
\newtheorem*{rems*}{Remarks}
\newtheorem*{rem*}{Remark}

\numberwithin{equation}{section}

\begin{document}

\title[Unimodular families of symmetric matrices] {Unimodular families of symmetric matrices }
\author{Wojciech Domitrz}
\address{Warsaw University of Technology\\ Faculty of Mathematics and Information
Science\\
ul. Koszykowa 75\\ 00-662 Warszawa\\ Poland}
\email{domitrz@mini.pw.edu.pl}
\author{Shyuichi Izumiya}
\address{Hokkaido University, Department of Mathematics\\ Faculty of Science\\ Kita 10 Nishi 8, Kita-ku\\
Sapporo\\ Japan}
\email{izumiya@math.sci.hokudai.ac.jp}
\thanks{}
\author{Hiroshi Teramoto}
\address{Hokkaido University\\ Research Institute for Electronic Science\\ Kita 20 Nishi 10, Kita-ku\\ Sapporo\\ Japan}
\address{PRESTO, Department of
Research Promotion\\ 7, Gobancho, Chiyodaku, Tokyo\\ Japan}
\email{teramoto@es.hokudai.ac.jp}
\thanks{The research was supported by JSPS KAKENHI Grant Number JP19K03484 and JST PRESTO Grant Number JPMJPR16E8.}
\date{}

\maketitle

\begin{abstract}
We introduce the volume-preserving equivalence a\-mong symmetric matrix-valued map-germs which is the unimo\-dular version of Bruce's $\mathcal{G}$-equivalence. The key concept to deduce unimo\-dular classification out of classification relative to
\begin{math}
\mathcal{G}
\end{math}-equivalence is symmetrical quasi-homogeneity, which is a generalization of the condition for a $2 \times 2$ symmetric matrix-valued map-germ in Corollary~2.1 (ii) by Bruce, Goryunov and Zakalyukin \cite{Bruce2002}. If a $\mathcal{G}$-equivalence class contains a symmetrically quasi-homogeneous representative, the class coincides with that relative to the volume-preserving equivalence (up to orientation reversing diffeomorphism in case if the ground field is real). By using that we show that all the simple classes relative to
\begin{math}
\mathcal{G}
\end{math}-equivalence in Bruce's list coincides with those relative to the volume preserving equivalence. Then, we classify map-germs from the plane to the set of $2 \times 2$ and $3 \times 3$ real symmetric matrices of corank at most $1$ and of
\begin{math}
\mathcal{G}_e
\end{math}-codimension less than $9$ and we show some of the normal forms split into two different unimodular singularities. We provide several examples to illustrate that non simplicity does not imply non symmetrical quasi-homogeneity and the condition that a map-germ is symmetrically quasi-homogeneous is stronger than one that each component of the map-germ is quasi-homogeneous. We also present an example of non symmetrically quasi-homogeneous normal form relative to
\begin{math}
\mathcal{G}
\end{math}
and its corresponding formal unimodular normal form.
\end{abstract}

\section{Introduction}
Semi-classical pseudo-differential operators in
\begin{math}
\mathbb{R}^d
\end{math}
whose principal symbols are
\begin{math}
n \times n
\end{math}-matrix valued functions on the phase space
\begin{math}
T^* \mathbb{R}^d
\end{math}
have seen a lot of applications in physics \cite{CV_I}. For example, Landau-Zener model describing non-adiabatic transitions through $n$ electron energy level (avoided) crossings can be described by the semi-classical pseudo-differential operators on
\begin{math}
\mathbb{R}^1
\end{math}
of the following form :
\begin{equation}
\hat{H} = i \hbar \frac{\partial}{\partial t} \otimes I_n - H \left( t \right)
\end{equation}
where
\begin{math}
\hbar
\end{math}
is the Planck constant that is supposed to be small,
\begin{math}
t
\end{math}
is a coordinate on
\begin{math}
\mathbb{R}^1
\end{math},
\begin{math}
I_n
\end{math}
is the
\begin{math}
n \times n
\end{math}
unit matrix, and
\begin{math}
H
\end{math}
is an
\begin{math}
n \times n
\end{math}
smooth matrix valued function of
\begin{math}
t
\end{math}
and non-adiabatic transitions can be understood in terms of the solution
\begin{math}
\psi \left( t \right) \in \mathbb{C}^n
\end{math}
of the following equation :
\begin{math}
\hat{H} \psi \left( t \right) = O \left( \hbar^\infty \right)
\end{math}.
In this case, its principal symbol can be written as
\begin{math}
\tau I_n - H \left( t \right)
\end{math}
by replacing the differential operator
\begin{math}
i \hbar \cfrac{\partial}{\partial t}
\end{math}
by the coordinate
\begin{math}
\tau
\end{math}
in
\begin{math}
T^*_t \mathbb{R}^1
\end{math}. This example is not only of physical interest but there is renewed interest from mathematicians for the cases
\begin{math}
n \ge 3
\end{math}
triggered by the discovery of a virtual turning point in WKB analysis of higher order differential equations \cite{Honda}.

Finding a normal form of the principal symbol of a semi-classical pseudo-differential operator is the first step toward constructing solutions of the corresponding equation. Braam and Duistermaat \cite{Braam} and Colin de Verdi\`ere \cite{CV_I} found normal forms in the case where a principal symbol is a real symmetric matrix under some generic conditions but their bifurcations are yet to be analyzed \cite{CV_I}. An ultimate goal of this paper is to take a step forward their analysis to non-generic cases and find normal forms for such cases.

In this paper, we classify map germs under an equivalence relation similar to that in \cite{Braam,CV_I} but symplectomorphisms are relaxed to volume-preserving diffeomorphisms. The equivalence relation is the unimodular version of Bruce's
\begin{math}
\mathcal{G}
\end{math}-equivalence. Bruce's $\mathcal G$-equivalence  is a special case of $\mathcal{K}$-equivalence
in the sense of Mather \cite{Mather1968,Mather1969}, which is also called $V$-equivalence in \cite{A}. Fortunately, $\mathcal G$ is a geometric subgroup of $\mathcal{K}$ in the sense of Damon \cite{Damon}. Therefore, the method of the singularity theory can work properly. But the group of symplectomorphism-germs at $0$ and the group of volume-preserving diffeomorphism-germs at $0$ are not geometric subgroups in the sense of Damon. Symplectomorphisms are too strict to get finitely-determined map germs away from the generic cases (see \cite{DJZ2} and \cite{D} for symplectic $V$-classification of algebraic varieties). In \cite{Martinet} on page 50 Martinet writes that the group of volume-preserving diffeomorphisms  is big enough that  there is
still some hope of finding a reasonable classification theorem in unimodular geometry. In fact Martinet was right. Results of Lando in \cite{Lando} and Varchenko in \cite{Varchenko} imply that if a hypersurface with isolated singularity in $\mathbb C^m$ is quasi-homogeneuos then its  unimodular $\mathcal{K}$-orbit coincides with its $\mathcal{K}$-orbit.   In \cite{DR} it is proved that  if an algebraic variety-germ is quasi-homogeneous, then its unimodular $\mathcal{K}$-orbit coincides with its (orientation-preserving) $\mathcal{K}$-orbit. It implies that the classifications of simple complete intersection singularities are identical for unimodular $\mathcal{K}$-equivalence and for $\mathcal{K}$-equivalence. Over $\mathbb R$  a $\mathcal{K}$-orbit  corresponds to one or two orbits in the volume preserving (hence orientation-preserving) case, otherwise the results are the same.

In case of
\begin{math}
d = 1 \; \left( r = 2 \right)
\end{math}, volume-preserving diffeomorphisms are symplectomorphisms and our classification provides classification of principal symbols. For example, the principal symbol of the semi-classical pseudo differential operator for the Landau-Zener model describing non-adiabatic transition through avoided crossing between two electron levels
\begin{equation}
\left(
\begin{matrix}
\tau + t & \epsilon \\
\epsilon & \tau - t
\end{matrix}
\right)
\end{equation}
can be understood as a unfolding of the class number $1$ of Table.~\ref{table:2}, where
\begin{math}
\epsilon
\end{math}
is the unfolding parameter. Our classification not only contains such a generic case but also non-generic cases and provides an insight for what types of bifurcations can occur.

The constitution of this paper is as follows. In Sec.~\ref{sec:vpsqh}, we introduce the volume-preserving equivalence among symmetric matrix-valued map-germs which is the unimodular version of Bruce's $\mathcal{G}$-equivalence. The key concept to deduce unimodular classification out of classification relative to
\begin{math}
\mathcal{G}
\end{math}-equivalence is symmetrical quasi-homogeneity, which is introduced in Definition~\ref{def:sqh}. If a $\mathcal{G}$-equivalence class contains a symmetrically quasi-homogeneous representative, the class coincides with that relative to the volume-preserving equivalence (up to orientation reversing diffeomorphism in case if the ground field is real). In Sec.~\ref{sec:sus}, we show that all normal forms in Bruce's list are symmetrically quasi-homogeneous, which indicate all the simple classes relative to
\begin{math}
\mathcal{G}
\end{math}-equi\-valence coincides with those relative to the volume preserving equivalence. In Sec.~\ref{sec:classification}, then, we classify map-germs from the plane to the set of $2 \times 2$ and $3 \times 3$ real symmetric matrices of corank at most $1$ and of
\begin{math}
\mathcal{G}_e
\end{math}-codimension less than $9$, which can be regarded as a real counterpart of Bruce's classification. The results are summarized in Table~\ref{table:2} and Table~\ref{table:n2m3}. In Sec.~\ref{sec:uni_normal_form}, we present the corresponding  unimodular normal forms in Table \ref{table:3} and Table \ref{table:Uni-n2m3}. In Sec.~\ref{sec:example}, we present several examples to illustrate the following: non simplicity does not imply non symmetrical quasi-homogeneity, the condition that a map-germ is symmetrically quasi-homogeneous is stronger than one that each component of the map-germ is quasi-homogeneous. We also present an example of non symmetrically quasi-homogeneous normal form relative to
\begin{math}
\mathcal{G}
\end{math}
and its corresponding formal unimodular normal form. Sec.~\ref{sec:conclusion} is devoted for the conclusion.

\section{Volume preserving equivalence and symmetrical quasi-homogeneity}\label{sec:vpsqh}
Let $M_n(\mathbb K)$ be the space of $n\times n$ matrices with the coefficients in the field $\mathbb K$ of
real or complex numbers.
Let $\textnormal{Sym}(n,\mathbb K)=\textnormal{Sym}_n$ be the subspace of $M_n(\mathbb K)$ of $n\times n$ symmetric matrices.  Let $\textnormal{GL}(n,\mathbb K)=\textnormal{GL}_n$ be the general linear group.
Let $C^{\infty}(\mathbb K^r)$ be the ring of smooth (or $\mathbb K$-analytic)  function-germs at $0$ on $\mathbb K^r$. Let $C^{\infty}(\mathbb K^r, S)$ be the space of  smooth (or $\mathbb K$-analytic)  map-germs $(\mathbb K^r,0)\rightarrow S$, where $S$ is a submanifold of  $M_n(\mathbb K)$. Let $\mathcal X(\mathbb K^r)$ denote the module of smooth (or $\mathbb K$-analytic)  vector field-germs at $0$ on $\mathbb K^r$. We denote by $A^T$ the transpose matrix of a matrix $A$.

\begin{defn}[\cite{B}]
The map-germs $A, B\in C^{\infty}(\mathbb K^r, \textnormal{Sym}_n)$  are $\mathcal{G}$-equivalent if there exists a diffeomorphism-germ $\Phi:   (\mathbb K^r,0)\rightarrow  (\mathbb K^r,0)$  and a  map-germ
$X \in C^{\infty}(\mathbb K^r, \textnormal{GL}_n)$ such that
$$B=X^T (A\circ \Phi) X.$$
\end{defn}
If $\Omega_0, \Omega_1$ are two volume form-germs at  $0$ on $\mathbb K^r$ and map-germs $A, B$ belong to $ C^{\infty}(\mathbb K^r, \textnormal{Sym}_n)$ then we define the following equivalence.

\begin{defn}
The pairs $(A,\Omega_0)$  and $(B,\Omega_1)$  are equivalent if there exists a diffeomorphism-germ $\Phi:   (\mathbb K^r,0)\rightarrow  (\mathbb K^r,0)$  and a  map-germ
$X \in C^{\infty}(\mathbb K^r, \textnormal{GL}_n)$ such that
$$B=X^T (A\circ \Phi) X, \ \ \Omega_1=\Phi^{\ast}\Omega_0.$$
\end{defn}

 Let us fix a volume form-germ   $\Omega$ at  $0$   on $\mathbb K^r$.
\begin{defn}
The map-germs $A, B\in C^{\infty}(\mathbb K^r, \textnormal{Sym}_n)$  are volume-preserving (or unimodular) equivalent if there exists a volume-pre\-serving diffeomorphism-germ $\Phi:   (\mathbb K^r,0)\rightarrow  (\mathbb K^r,0)$ (i. e. $\Phi^{\ast}\Omega=\Omega $) and a  map-germ
$X \in C^{\infty}(\mathbb K^r, \textnormal{GL}_n)$ such that
$$B=X^T (A\circ \Phi) X.$$
\end{defn}

A  diffeomorphism-germ $\Phi:   (\mathbb R^r,0)\rightarrow  (\mathbb R^r,0)$ is orientation-preserving (orientation-reversing resp.) if  $\det d\Phi|_0>0$ ($\det d\Phi|_0<0$ resp.).

If $\mathbb K=\mathbb R$ we also define an orientation-preserving equivalence.

\begin{defn}
The map-germs $A, B\in C^{\infty}(\mathbb R^r, \textnormal{Sym}_n)$  are orien\-tation-preserving  equivalent if there exists an orientation-preserving  diffeomorphism-germ $\Phi:   (\mathbb R^r,0)\rightarrow  (\mathbb R^r,0)$ and a  map-germ
$X \in C^{\infty}(\mathbb R^r, \textnormal{GL}_n)$ such that
$$B=X^T (A\circ \Phi) X.$$
\end{defn}

Let $\Lambda^k$ denote the space (of germs) of smooth (or $\mathbb K$-analytic)
differential $k$-forms on $(\mathbb K ^r,0)$, and denote
the subset of $\Lambda ^r$ of (germs of)  volume forms
by $\mbox{Vol}$. Let $\mathcal D$ denote the group of diffeomorphism-germs $ (\mathbb K^r,0)\rightarrow  (\mathbb K^r,0)$. For a given subgroup $H$ of the group of diffeomorphism-germs of $\mathcal D$ we consider
a $C^{\infty}(\mathbb K^r)$-module $M$ in the Lie algebra $LH$ of $H$
(and $M=LH$ if $LH$ itself is a $C^{\infty}(\mathbb K^r)$-module). In the following
$\Omega$ and $\Omega _i$ always denote volume form-germs
in $(\mathbb K ^r,0)$.

\begin{defn}\label{H-diff}
We say that $\Omega _0$ and $\Omega _1$ are $H$-diffeomorphic if
there is a diffeomorphism-germ $\Phi\in H$ such that
$\Phi^{\ast}\Omega_1=\Omega_0$
\end{defn}

\begin{defn}\label{IH-diff}
We say that $\Omega _0$ and $\Omega _1$ are $H$-isotopic if there
is a smooth (or $\mathbb K$-analytic)  family of diffeomorphism-germs $\Phi_t\in H$ for $t\in
[0,1]$ such that $\Phi_1^{\ast}\Omega_1=\Omega_0$ and
$\Phi_0=\mbox {Id}$.
\end{defn}

\begin{defn}\label{M-eq}
We say that $\Omega _0$ and $\Omega _1$ are $M$-equivalent if
there is a vector field-germ $V\in M$ such that $\Omega _0-\Omega
_1=d(V\rfloor \Omega)$ (for any volume form-germ $\Omega$).
\end{defn}

\begin{thm}[Theorem 2.6 in \cite{DR}] \label{thm2.8}
\label{H-Moser} If $\Omega _0$ and $\Omega _1$ are $M$-equivalent
volume form-germs, which for $\mathbb K=\mathbb R$ define the same
orientation, then $\Omega _0$ and $\Omega _1$ are $H$-isotopic.
\end{thm}

We denote by $\mathcal D_A$ the following set
$$
\{ \Phi\in \mathcal D| \exists  X \in C^{\infty}(\mathbb K^r, \textnormal{GL}_n)  \ \ A=X^T (A\circ \Phi) X\}.
$$
 It is easy to see that $\mathcal D_A$ is a subgroup of $\mathcal D$. It is the stabilizer (the isotropy group) of $A$.
The Lie algebra $L\mathcal D_A$ of the isotropy group $\mathcal D_A$   has the following form
$$
L\mathcal D_A=\{ V\in \mathcal X(\mathbb K^r)| \exists U\in C^{\infty}(\mathbb K^r,M_{n}(\mathbb K)) \  dA(V)=U^T A+AU\}.
$$

We prove the following properties of $L\mathcal D_A$.

\begin{lem}\label{Lem}
Let $A$ belong to $C^{\infty}(\mathbb K^r, \textnormal{Sym}_n)$.
Then $L\mathcal D_A$ is a $C^{\infty}(\mathbb K^r)$-module.
If a vector field-germ $V$ belongs to $L\mathcal D_A$ then $V$ belongs to  $L\mathcal D_{X^TAX}$ for any map-germ $X\in C^{\infty}(\mathbb K^r, \textnormal{GL}_n)$.
\end{lem}

\begin{proof}
Let $V_1, V_2 \in L\mathcal D_A$ and $f\in C^{\infty}(\mathbb K^r)$. Then there exists a map-germ $U_i\in C^{\infty}(\mathbb K^r,M_{n}(\mathbb K))$ such that $dA(V_i)=U_i^T A+AU_i$ for $i=1,2$. Thus
$dA(V_1+fV_2)=dA(V_1)+fdA(V_2)=U_1^T A+AU_1+f(U_2^T A+AU_2)=(U_1+fU_2)^TA+A(U_1+fU_2)$. It implies that $V_1+fV_2\in  L\mathcal D_A$. Hence $L\mathcal D_A$ is a $C^{\infty}(\mathbb K^r)$-module.

Since $V\in L\mathcal D_A$ there exists a map-germ $U\in C^{\infty}(\mathbb K^r,M_{n}(\mathbb K))$ such that $dA(V)=U^T A+AU$.
Then
\begin{multline}
d(X^T A X)(V)=(dX(V))^T A X+X^T dA(V) X+X^T A dX(V) \\
= (dX(V))^T  ((X^{-1})^T X^T) A X+X^T (U^T A+A U) X+X^T A ( X X^{-1}) dX(V) \\
= ( X^{-1} dX(V) +X^{-1}UX  )^T(X^T A X)+(X^T A  X) (X^{-1} dX(V)+X^{-1}U X).
\end{multline}

It implies that for $C=X^{-1} dX(V) +X^{-1}UX$ we have
$$ d(X^T A X)(V)=C^T(X^T A X)+(X^T A X) C. $$ Thus $V$ belongs to $L\mathcal D_{X^TAX}$.
\end{proof}

\begin{defn} \label{def:sqh} A smooth (or $\mathbb K$-analytic)  map-germ $A$,  which belongs to  $ C^{\infty}(\mathbb K^r, \textnormal{Sym}_n)$, is symmetrically quasi-homogeneous if there exists a coordinate system $(x_1,\cdots,x_r)$ on $(\mathbb K^r,0)$, non-negative integers $\lambda_k$ for $k=1,\cdots,r$ and $\delta_i$ for $i=1,\cdots,n$, a map-germ $X \in C^{\infty}(\mathbb K^r, \textnormal{GL}_n)$ such that $\sum_{k=1}^r \lambda_k>0$  and
 $$ B_{ij}(t^{\lambda_1}x_1,\cdots, t^{\lambda_r}x_r)=t^{\frac{1}{2}(\delta_i+\delta_j)}B_{ij}(x_1,\cdots,x_r)$$
 for any $t\in \mathbb K$ and $i,j=1,\cdots, n$, where $B=X^T A X$.

 The integers $\lambda_1,\cdots, \lambda_r$ are called weights and $\delta_1,\cdots, \delta_n$ are called quasi-degrees.
\end{defn}

\begin{defn}\label{gen-Eul}
A linear vector field
$$
E_{\lambda}=\sum_{i=1}^r \lambda_i x_i\frac{\partial}{\partial x_i}.
$$
with integer coefficients $\lambda_i$ is called a generalized
Euler vector field (for coordinates
$(x_1,\ldots,x_r)\in\mathbb K ^r$ and weights $(\lambda_1,\ldots,\lambda_r)$). The number $\sum_{i=1}^r \lambda_i$ is  called a total weight.
\end{defn}

If  a map-germ $A \in C^{\infty}(\mathbb K^r, \textnormal{Sym}_n)$ is symmetrically quasi-homo\-geneous with weights  $(\lambda_1,\ldots,\lambda_r)$ and quasi-degrees $(\delta_1,\cdots, \delta_n)$ in the coordinate system $(x_1,\cdots,x_r)$ on $(\mathbb K^r,0)$
 then for any $i,j=1,\cdots,n$
$$
dA_{ij}(E_{\lambda})=\frac{\delta_i+\delta_j}{2}A_{ij}.
$$
Thus $dA(E_{\lambda})=\left(\frac{1}{2}\text{diag}(\delta_1,\cdots,\delta_n)\right)^TA+A \left(\frac{1}{2}\text{diag}(\delta_1,\cdots,\delta_n)\right).$

Thus by Lemma \ref{Lem} we obtain the following proposition (see also \cite{DJZ1}).
\begin{prop}
If  a map-germ $A \in C^{\infty}(\mathbb K^r, \textnormal{Sym}_n)$ is symmetrically quasi-homogeneous then the Euler vector field $E_{\lambda}$ belongs to $L\mathcal D_A$.
\end{prop}

The following proposition is crucial for our considerations.
\begin{prop}[Proposition 2.13 in \cite{DR}]\label{prEuler}
Let $V$ be a smooth (or $\mathbb K$-analytic)  vector field-germ
on $(\mathbb K^r,0)$ which is locally diffeomorphic to a
generalized Euler vector field-germ with non-negative weights
and   positive total weight.
If $V$ generates a $C^{\infty}(\mathbb K^r)$-module in $LH$ then any two
volume form-germs (which over $\mathbb K =\mathbb R$ define the same orientation)
are $H$-isotopic.
\end{prop}

By Proposition \ref{prEuler} we prove the following result. 

\begin{thm}\label{qh-main}
If a smooth (or $\mathbb K$-analytic)  map-germ $A$, which belongs to  $C^{\infty}(\mathbb K^r, \textnormal{Sym}_n)$, is symmetrically quasi-homogeneous then any two
germs of volume forms (which over $\mathbb K =\mathbb R$ define the same orientation)
are $\mathcal D_A$-isotopic.

\end{thm}

\begin{proof}
By Lemma \ref{Lem} $L\mathcal D_A$ is a $C^{\infty}(\mathbb K^r)$-module. The  Euler vector field $E_{\lambda}$ belongs to $L\mathcal D_A$ and it generates $C^{\infty}(\mathbb K^r)$-module.
By Proposition \ref{prEuler} any two
germs of volume forms (which over $\mathbb K =\mathbb R$ define the same orientation)
are $\mathcal D_A$-isotopic.
\end{proof}

Let us fix a volume form-germ $\Omega$ at $0$ in $\mathbb K^r$.
\begin{cor}\label{C-eq}
Let  $A \in C^{\infty}(\mathbb C^r, \textnormal{Sym}_n)$ be    symmetrically quasi-homo\-geneous.
If $B  \in C^{\infty}(\mathbb C^r, \textnormal{Sym}_n)$ is $\mathcal G$-equivalent to $A$ then $B$ is volume-preserving  equivalent to $A$.
\end{cor}

\begin{prop}\label{split}
Let  $A  \in C^{\infty}(\mathbb R^r, \textnormal{Sym}_n)$  be    symmetrically quasi-homogeneous.

If  there exits a diffeomorphism-germ $\Phi \in \mathcal D_A$ reversing the orientation of $\mathbb R^r$ then any $B  \in C^{\infty}(\mathbb R^r, \textnormal{Sym}_n)$,  which is  $\mathcal G$-equivalent to $A$,  is volume-preserving  equivalent to $A$.

Let us assume that   every diffeomorphism-germ $\Phi \in \mathcal D_A$ preserves the orientation of $\mathbb R^r$.

If   $B  \in C^{\infty}(\mathbb R^r, \textnormal{Sym}_n)$ is  orientation-preserving equivalent to $A$  then $B$ is volume-preserving  equivalent to $A$.

If $B  \in C^{\infty}(\mathbb R^r, \textnormal{Sym}_n)$ is equivalent to $A$ but $B$ is not orientation-preserving equivalent to $A$ then $B$ is volume-preserving  equivalent to $\bar{A}$,
where $$\bar{A}(x_1,x_2,\cdots,x_r)=A(-x_1,x_2,\cdots,x_r),$$
and $A, \bar{A}$ are not volume-preserving equivalent.
\end{prop}

\begin{proof}
 First let us assume that  there exits a diffeomorphism-germ $\Phi \in \mathcal D_A$ reversing the orientation of $\mathbb R^r$. If  $B  \in C^{\infty}(\mathbb R^r, \textnormal{Sym}_n)$  is  equivalent to $A$ then  there exists a diffeomorphism-germ $\Psi:   (\mathbb R^r,0)\rightarrow  (\mathbb R^r,0)$  and a  map-germ
$X \in C^{\infty}(\mathbb R^r, \textnormal{GL}_n)$ such that $B\circ \Psi=X^T A X$. Then the pair  $(B,\Omega)$ is equivalent to the pair$(A,\Psi^{\ast} \Omega)$. If $\Omega$ and $\Psi^{\ast}\Omega$ define different orientations of $\mathbb R^r$ then we pullback $\Psi^{\ast}\Omega$ by the orientation-reversing diffeomorphism-germ $\Phi \in \mathcal  D_A$.  Thus we may assume that the two volume forms  define the same orientation.  By Theorem \ref{qh-main} we obtain that  the pair$(A,\Psi^{\ast} \Omega)$ is equivalent to $(A,\Omega)$, which implies that $B$ and $A$ are volume-preserving equivalent.

Let us assume that   every diffeomorphism-germ $\Phi \in \mathcal D_A$ preserves the orientation of $\mathbb R^r$.
If  $B$ is a orientation-preserving equivalent to $A$ then there exists a orientation-preserving diffeomorphism-germ $\Psi:   (\mathbb R^r,0)\rightarrow  (\mathbb R^r,0)$  and a  map-germ  $X \in C^{\infty}(\mathbb R^r, \textnormal{GL}_n)$ such that $B\circ \Psi=X^T A X.$ Thus $\Psi^{\ast}\Omega$ and $\Omega$ define the same orientation. By Theorem \ref{qh-main} there exists $\Phi \in \mathcal D_A$ such that $\Phi^{\ast}(\Psi^{\ast} \Omega)=\Omega$. Thus $B$ and $A$  are volume-preserving equivalent.
If  there exists a orientation-reversing diffeomorphism-germ $\Psi:   (\mathbb R^r,0)\rightarrow  (\mathbb R^r,0)$  and a  map-germ
$X \in C^{\infty}(\mathbb R^r, \textnormal{GL}_n)$ such that $B\circ \Psi=X^T A X$ than $\Psi^{\ast}\Omega$ and $-\Omega$ define the same orientation. Thus by Theorem \ref{qh-main} there exists $\Phi \in \mathcal D_A$ such that $\Phi^{\ast}(\Psi^{\ast} \Omega)=-\Omega$. If $I(x_1,x_2,\cdots, x_r)=(-x_1,x_2,\cdots,x_r)$ then $I^{\ast}\Omega=-\Omega$ and $\bar{A}\circ I=A$. It implies that $B$ and $\bar{A}$  are volume-preserving equivalent.

If $A$ and $\bar{A}$  are volume-preserving equivalent then  $\Omega$ and $-\Omega$  are $\mathcal D_A$-equivalent, which is impossible since every diffeomorphism-germ $\Phi \in \mathcal D_A$ is orientation-preserving.

\end{proof}

In general, it is difficult to check if a given map-germ 
\begin{math}
A \in C^{\infty} \left( \mathbb{K}^r, \textnormal{Sym}_n \right)
\end{math}
is symmetrically quasi-homogeneous. In what follows, we provide two useful criteria for that.

If 
\begin{math}
L \mathcal{D}_A
\end{math}
does not contain a vector field such that the sum of eigenvalues of the linear part of the vector field is positive, 
\begin{math}
A
\end{math}
cannot be symmetrically quasi-homogeneous. This can be shown as follows: Suppose there exists a diffeomorphism-germ
\begin{math}
\Phi \colon \left( \mathbb{K}^r, 0 \right) \rightarrow \left( \mathbb{K}^r, 0 \right)
\end{math}
and a map-germ 
\begin{math}
X \in C^{\infty} \left( \mathbb{K}^r, \textnormal{GL}_n \right)
\end{math}
such that 
\begin{math}
X^T \left( A \circ \Phi \right) X
\end{math}
is symmetrically quasi-homogeneous. Then, there exists a generalized Eular vector fields 
\begin{math}
E
\end{math}
such that 
\begin{math}
E \in L \mathcal{D}_{X^T \left( A \circ \Phi \right) X}
\end{math}
holds. By using Lemma~\ref{Lem}, this implies that 
\begin{math}
E \in L \mathcal{D}_{A \circ \Phi}
\end{math}
and thus 
\begin{math}
\Phi_* E \in L \mathcal{D}_A
\end{math}
hold. Since 
\begin{math}
\Phi_*
\end{math}
keeps the eigenvalues of the linear part of the vector field 
\begin{math}
E
\end{math}
invariant, the sum of eigenvalues of the linear part of the vector field 
\begin{math}
\Phi_* E
\end{math}
is positive. This proves the claim.

If 
\begin{math}
\mathbb{K} = \mathbb{C}
\end{math}
and a map-germ 
\begin{math}
A
\end{math}
is holomorphic, there exists another useful criterion based on Corollary~1.8 and Corollary~2.1 by Bruce, Goryunov and Zakalyukin \cite{Bruce2002}. 
\begin{thm}\label{thm2.7}
Suppose the map-germ 
\begin{math}
A
\end{math}
is holomorphic, symmetrically quasi-homogeneous, has finite
\begin{math}
\mathcal{G}_e
\end{math}-codimension and the function-germ 
\begin{math}
\left( \det \circ A \right) \colon \left( \mathbb{C}^r, 0 \right) \rightarrow \left( \mathbb{C}, 0 \right)
\end{math}
has an isolated singularity at the origin. Let $\mu \left( \det \circ A \right)$ be the Milnor number of the function-germ at the origin. Then, 
\begin{equation}
\mu \left( \det \circ A \right) = \mathcal{G}_e\textnormal{-codim} \left( A \right) - \beta_1 + \beta_0
\end{equation}
holds where 
\begin{math}
\beta_j \; \left( j = 0, 1 \right)
\end{math}
are the $j$-th Betti numbers of the Koszul complex of the ideal generated by 
\begin{math}
\left( n-1 \right) \times \left( n-1 \right)
\end{math}
minors of
\begin{math}
A
\end{math}.
\end{thm}
In this case, 
\begin{math}
\mathcal{G}_e
\end{math}-codimension of 
\begin{math}
A
\end{math}
coincides with the Tjurina number 
\begin{math}
\tau_V \left( A \right)
\end{math}
defined in \cite{Bruce2002} where 
\begin{math}
V = \left\{ B \in \textnormal{Sym}_n \left| \det B = 0 \right. \right\}
\end{math}. This theorem is a direct consequence of Corollary~1.8 in \cite{Bruce2002}.

There is a characterization of quasi-homogeneous holomorphic function germs by Kyoji Saito \cite{Saito1971}. For a holomorphic function germ $f$ with isolated singularities, $f$ is quasi-homogeneous if and only if $\mu (f)=\tau (f), $ where $\mu (f)$ is the Milner number and $\tau(f)$ is the Tyrina number of $f,$ respectively. Theorem~\ref{thm2.7} gives a necessary condition for symmetrically quasi-homogeneous symmetric matrix valued map germs similar to the Saito's characterization of quasi-homogeneous function germs.

\section{Simple unimodular singularities}\label{sec:sus}

In \cite{B} J. W. Bruce obtained the list of $\mathcal G$-simple singularities of families of symmetric matrices. We show that all normal forms in Bruce's list are symmetrically quasi-homogeneous.
\begin{prop}\label{SymQH}
All Bruce's $\mathcal G$-simple singularities of  families of symmetric matrices  are symmetrically quasi-homogeneous.
\end{prop}
\begin{proof}

Let $A \colon (\mathbb C^r,0)\rightarrow \textnormal{Sym}_n$ be a $\mathcal G$-simple germ of rank $0$ at the origin  from Bruce's list (\cite{B}).

If $r=1$ then $A=\text{diag}(x^{m_1},x^{m_2},\cdots,x^{m_n})$, where $m_1\le m_2 \le \cdots \le m_n$. Then $A$ is symmetrically quasi-homogeneous with a weight $\lambda_1=1$ and quasi-degrees $\delta_i=m_i$ for $i=1,\cdots,n$.

When the corank of $dA(0)$ is $0$ then a normal form of a $\mathcal G$-simple germ $A \colon \mathbb C^N\times \mathbb C^s\rightarrow \textnormal{Sym}_n$ is given by $A_{ij}(x, z)=x_{ij}$ for $i,j=1,\cdots,n$, where $N=\frac{n(n+1)}{2}$. Then $A$ is symmetrically quasi-homogeneous with weights $\lambda_{ij}=1$ for $i, j=1,\cdots,n$, $\lambda_{N+k}=1$ for $k=1,\cdots,s$ and quasi-degrees $\delta_i=1$ for $i=1,\cdots,n$.

When the corank of $dA(0)$ is $1$ there are two cases:

A normal form a $\mathcal G$-simple germ $A \colon \mathbb C^{N-1}\times \mathbb C^s\rightarrow \textnormal{Sym} _n$ given by $A_{ij}(x, z)=x_{ij}$ for $(i,j)\ne (1,1)$ and $A_{11}(x,z)=\sum_{i=2}^n \epsilon_i x_{ii}+f(z)$, where
$f \colon (\mathbb C^s,0) \rightarrow \mathbb C$ is one of Arnold's $\mathcal R$-simple germs (\cite{A}). Normal forms of Arnold's $\mathcal R$-simple germs are quasi-homogeneous. Let us assume that $f$ is quasihomogeneous with weights $w_1,\cdots,w_s$ and the quasi-degree $\delta$. Then $A$ is symmetrically quasi-homogeneous with weights $\lambda_{ij}=\delta$ for $(i,j)\ne (1,1)$, $\lambda_{N-1+k}=w_k$ for $k=1,\cdots,s$ and quasi-degrees $\delta_i=\delta$ for $i=1,\cdots,n$.

A normal form a $\mathcal G$-simple germ $A \colon \mathbb C^{N-1}\times \mathbb C^s\rightarrow \textnormal{Sym}_n$ is given by $A_{ij}(x, z)=x_{ij}$ for $(i,j)\ne (1,1)$ and $A_{11}(x,z)=\sum_{i=2}^{n-1} \epsilon_i x_{ii}+f(x_{nn},z)$, where
$f \colon (\mathbb C\times \mathbb C^s,0) \rightarrow \mathbb C$ is one of Arnold's simple germs of functions on manifolds with boundary $\{(x_{nn},z)|x_{nn}=0\}$ (\cite{A}). Normal forms of Arnold's simple germs of functions on manifolds with boundary are quasi-homogeneous. Let us assume that $f$ is quasi-homogeneous with weights $w_0, z_1,\cdots, z_s$ and the quasi-degree $\delta$. Then $A$ is symmetrically quasi-homogeneous with weights $\lambda_{ij}=\delta$ for $i, j=1,\cdots,n-1$, $(i,j)\ne (1,1)$, $\lambda_{nn}=w_0$, $\lambda_{ni}=\lambda_{in}=1/2(w_0+\delta)$ for $i=1,\cdots,n-1$, $\lambda_{N-1+k}=w_k$ for $k=1,\cdots,s$ and quasi-degrees $\delta_i=\delta$ for $i=1,\cdots,n-1$, $\delta_n=w_0$.

The $\mathcal G$-simple germs  $(\mathbb C^r,0)\rightarrow \textnormal{Sym}_n$  for $r=n=2$, $r=2,\ n=3$ and $r=4, \ n=3$  are presented in Tables \ref{tableB1}-\ref{tableB3}. All of them are symmetrically quasi-homogeneous with weights $\lambda_i$ for $i=1,\cdots,r$ and quasi-degrees $\delta_j$ for $j=1,\cdots,n$ presented in the tables.

\begin{center}
\begin{table}[h]

    \begin{small}
    \noindent
    \begin{tabular}{|c|l|c|c|c|c|}

            \hline
BN  &  normal form & $\lambda_1$ & $\lambda_2$  & $\delta_1$ & $\delta_2$ \\  \hline
1 & $\left(\begin{array}{cc}
     x_2^k & x_1 \\
     x_1     & x_2^l
     \end{array}\right)$, $k\ge 1$, $l\ge 2$  & $k+l$ & $2$ & $2k$ &$2l$ \\
2 & $\left(\begin{array}{cc}
     x_1 & 0 \\
     0     & x_2^2+x_1^k
     \end{array}\right)$, $k\ge 2$  & $2$ & $k$ & $2$ &$2k$ \\
3 & $\left(\begin{array}{cc}
     x_1 & 0 \\
     0     & x_1x_2+x_2^k
     \end{array}\right)$, $k\ge 2$  & $k-1$ & $1$ & $k-1$ &$k$ \\
4 & $\left(\begin{array}{cc}
     x_1 & x_2^k \\
     x_2^k     & x_1x_2
     \end{array}\right)$, $k\ge 2$  & $2k-1$ & $2$ & $2k-1$ &$2k+1$ \\
5 & $\left(\begin{array}{cc}
     x_1 & x_2^2 \\
     x_2^2    & x_1^2
     \end{array}\right)$ & $4$ & $3$ & $4$ &$8$ \\
6 & $\left(\begin{array}{cc}
     x_1 & 0 \\
     0     & x_1^2+x_2^3
     \end{array}\right)$  & $6$ & $4$ & $6$ &$12$ \\   \hline
\end{tabular}

\smallskip

\caption{\small $\mathcal G$-simple singularities  $(\mathbb C^2,0)\rightarrow \textnormal{Sym}_2$. $\lambda_1$, $\lambda_2$ are weigths and $\delta_1$, $\delta_2$ are quasi-degrees of symmetrical quasi-homogeneity.}\label{(2,2)}
 \label{tableB1}
\end{small}
\end{table}
\end{center}

\begin{center}
\begin{table}[h]

    \begin{small}
    \noindent
    \begin{tabular}{|c|l|c|c|c|c|c|}

            \hline
BN &   normal form & $\lambda_1$ & $\lambda_2$  & $\delta_1$ & $\delta_2$ & $\delta_3$ \\  \hline
1 & $\left(\begin{array}{ccc}
     x_2^k & x_1   & 0\\
     x_1     & x_2^l& 0\\
     0     & 0   & x_2
     \end{array}\right)$, $k\ge 1$, $l\ge 2$  & $k+l$ & $2$ & $2k$ &$2l$ & $2$ \\
2 &$\left(\begin{array}{ccc}
     0 & x_1 & x_2\\
     x_1 & x_2 & 0\\
     x_2 & 0 & x_1^2
     \end{array}\right)$  & $3$ & $4$ & $2$ &$4$ & $6$ \\
3 &  $\left(\begin{array}{ccc}
     0 & x_1 & x_2\\
     x_1 & x_2 & 0\\
     x_2 & 0 & x_1x_2
     \end{array}\right)$  & $2$ & $3$ & $1$ &$3$ & $5$ \\
4 &   $\left(\begin{array}{ccc}
     0 & x_1 & x_2\\
     x_1 & x_2 & 0\\
     x_2 & 0 & x_1^3
     \end{array}\right)$  & $3$ & $5$ & $1$ &$5$ & $9$ \\
 5 &   $\left(\begin{array}{ccc}
     x_1 & 0 & 0\\
     0 & x_2 & x_1\\
     0 & x_1 & x_2^2
     \end{array}\right)$  & $3$ & $2$ & $3$ &$2$ & $4$ \\
6 &     $\left(\begin{array}{ccc}
     x_1 & 0 & x_2^2\\
     0 & x_2 & x_1\\
     x_2^2 & x_1 & 0
     \end{array}\right)$  & $5$ & $3$ & $5$ &$3$ & $7$ \\   \hline
\end{tabular}

\smallskip

\caption{\small $\mathcal G$-simple singularities  $(\mathbb C^2,0)\rightarrow \textnormal{Sym}_3$. $\lambda_1$, $\lambda_2$ are weigths and $\delta_1$, $\delta_2$, $\delta_3$ are quasi-degrees of symmetrical quasi-homogeneity.}\label{(2,3)}
\label{tableB2}
\end{small}
\end{table}
\end{center}

\begin{center}
\begin{table}[h]

    \begin{small}
    \noindent
    \begin{tabular}{|c|l|c|c|c|c|c|c|c|}

            \hline
 BN &   normal form & $\lambda_1$ & $\lambda_2$  & $\lambda_3$ & $\lambda_4$  & $\delta_1$ & $\delta_2$ & $\delta_3$ \\  \hline
1 &$\left(\begin{array}{ccc}
     x_1 & 0   & x_3\\
     0 & x_2+x_1^k& x_4\\
     x_3     & x_4   & -x_2
     \end{array}\right)$, $k\ge 1$   & $2$ & $2k$ & $k+1$ &$2k$ & $2$ & $2k$ &$2k$\\
2 &$\left(\begin{array}{ccc}
     x_1 & x_4^2 & x_2\\
     x_4^2 & -x_2 & x_3\\
     x_2 & x_3 & x_4
     \end{array}\right)$  & $7$ & $5$ & $4$ &$3$ & $7$ & $5$ & $3$\\
3 & $\left(\begin{array}{ccc}
     x_1 & x_4 x_3 & x_2\\
     x_4 x_3 & -x_2 & x_3\\
     x_2 & x_3 & x_4
    \end{array}\right)$  & $6$ & $4$ & $3$ &$2$ & $6$ & $4$ & $2$ \\
 4 &     $\left(\begin{array}{ccc}
     x_1 & x_4^3 & x_2\\
     x_4^3 & -x_2 & x_3\\
     x_2 & x_3 & x_4
     \end{array}\right)$  & $11$ & $7$ & $5$ &$3$ & $11$ & $7$ & $3$\\   \hline

\end{tabular}

\smallskip

\caption{\small $\mathcal G$-simple singularities  $(\mathbb C^4,0)\rightarrow \textnormal{Sym}_3$. $\lambda_1,\cdots,\lambda_4$ are weigths and $\delta_1$, $\delta_2$, $\delta_3$ are quasi-degrees of symmetrical quasi-homogeneity. }\label{(4,3)}
\label{tableB3}
\end{small}
\end{table}
\end{center}

Thus all Bruce's $\mathcal G$-simple germs are symmetrically quasi-homogeneous.
\end{proof}
By Proposition \ref{SymQH} and Corolary \ref{C-eq} we obtain the following proposition.
\begin{prop}
All Bruce's $\mathcal G$-simple singularities of  families of symmetric matrices  are simple unimodular singularities.
\end{prop}

\section{Classification of map-germs $\left( \mathbb{R}^r, 0 \right) \rightarrow \left( \textnormal{Sym}_n, O_n \right)$ of corank at most $1$ for $r = 2, n = 2, 3$} \label{sec:classification}
In this section we classify map-germs
\begin{math}
\left( \mathbb{R}^r, 0 \right) \rightarrow \left( \textnormal{Sym}_n, O_n \right)
\end{math}
of corank at most $1$ for the cases $r = 2, n = 2, 3$ of
\begin{math}
\mathcal{G}_e
\end{math}-codimension less than $9$, where
\begin{math}
O_n
\end{math}
is the $0 \times 0$ zero matrix. The results of this section are summarized in Table~\ref{table:2} and Table~\ref{table:n2m3}.

We first start with classification of $1$-jets for
\begin{math}
r = 2, n = 2
\end{math}
of corank at most $1$. We denote a coordinate in
\begin{math}
\mathbb{R}^2
\end{math}
as
\begin{math}
x = \left( x_1, x_2 \right)
\end{math}. Classification for
\begin{math}
r = 2, n = 3
\end{math}
can be done in the same manner and we omit the proof for the case.

For
\begin{math}
A \colon \left( \mathbb{R}^2, 0 \right) \rightarrow \left( \textnormal{Sym}_2, O_2 \right)
\end{math}, its
\begin{math}
1
\end{math}-jet can be written as
\begin{math}
j^1 A = C x_1 + D x_2
\end{math}
where
\begin{math}
C, D \in \textnormal{Sym}_2 \left( \mathbb{R} \right)
\end{math}. By an appropriate action of
\begin{math}
\mathcal{H}
\end{math}, the $1$-jet can be transformed into one of the following forms by using the theory of matrix pencil (\textbf{Theorem~9.2} in \cite{Lancaster2005}):
\begin{enumerate}
\item
\begin{math}
\left(
\begin{matrix}
0 & 0 \\
0 & 0 \\
\end{matrix} \right)
\end{math}, \label{item1}
\item
\begin{math}
\left(
\begin{matrix}
0 & 0 \\
0 & \delta_1
\end{matrix} \right) x_1
\end{math} (\begin{math}
\delta = \pm 1
\end{math}), \label{item2}
\item
\begin{math}
\left(
\begin{matrix}
0 & 0 \\
0 & \delta_1 c_1
\end{matrix} \right) x_1 + \left(
\begin{matrix}
0 & 0 \\
0 & \delta_1
\end{matrix} \right) x_2
\end{math} (\begin{math}
c_1 \in \mathbb{R}, \delta_1 = \pm 1
\end{math}), \label{item3}
\item
\begin{math}
\delta_1 \left(
\begin{matrix}
0 & 1 \\
1 & 0
\end{matrix} \right) x_1 + \delta_1 \left(
\begin{matrix}
1 & 0 \\
0 & 0
\end{matrix} \right) x_2
\end{math} (\begin{math}
\delta_1 = \pm 1
\end{math}), \label{item4}
\item
\begin{math}
\left(
\begin{matrix}
\delta_1 & 0 \\
0 & \delta_2
\end{matrix} \right) x_1
\end{math}, (\begin{math}
\delta_1, \delta_2 = \pm 1
\end{math}) \label{item5}
\item
\begin{math}
\left(
\begin{matrix}
\delta_1 & 0 \\
0 & \delta_2 c_1
\end{matrix} \right) x_1 + \left(
\begin{matrix}
0 & 0 \\
0 & \delta_2
\end{matrix} \right) x_2
\end{math} (\begin{math}
c_1 \in \mathbb{R}, \delta_1, \delta_2 = \pm 1
\end{math}), \label{item6}
\item
\begin{math}
\delta_1 \left(
\begin{matrix}
1 & c_1 \\
c_1 & 0
\end{matrix} \right) x_1 + \delta_1 \left(
\begin{matrix}
0 & 1 \\
1 & 0
\end{matrix} \right) x_2
\end{math} (\begin{math}
c_1 \in \mathbb{R}, \delta_1 = \pm 1
\end{math}), \label{item7}
\item
\begin{math}
\left(
\begin{matrix}
\delta_1 c_1 & 0 \\
0 & \delta_2 c_2
\end{matrix} \right) x_1 + \left(
\begin{matrix}
\delta_1 & 0 \\
0 & \delta_2
\end{matrix} \right) x_2
\end{math} (\begin{math}
c_1, c_2 \in \mathbb{R}, \delta_1, \delta_2 = \pm 1
\end{math}), \label{item8}
\item
\begin{math}
\left(
\begin{matrix}
c_1 & c_2 \\
c_2 & -c_1
\end{matrix} \right) x_1 + \left(
\begin{matrix}
0 & 1 \\
1 & 0
\end{matrix} \right) x_2
\end{math}
\begin{math}
\left( c_1, c_2 \in \mathbb{R}, c_1 \neq 0 \right)
\end{math} \label{item9}
\end{enumerate}
By changing coordinate system in
\begin{math}
\mathbb{R}^2
\end{math}, the above $9$ cases can be reduced to the $6$ cases in Table~\ref{table:1}.
\begin{table}[h]
 \begin{center}
  \begin{tabular}{|l|l|l|l|} \hline
   class num & $j^1 A$ & $j^2 \det A$ & $\textnormal{rank} \; dA \left( 0 \right)$ \\ \hline
   1 & $\left( \begin{matrix} 0 & 0 \\ 0 & 0 \end{matrix} \right)$ & $0$ & $0$ \\
   2 & $\left( \begin{matrix} x_1 & 0 \\ 0 & 0 \end{matrix} \right)$ & $0$ & $1$ \\
   3 & $\left( \begin{matrix} x_2 & x_1 \\ x_1 & 0 \end{matrix} \right)$ & $-x_1^2$ & $2$ \\
   4 & $\left( \begin{matrix} x_1 & 0 \\ 0 & x_2 \end{matrix} \right)$ & $x_1 x_2$ & $2$ \\
   5 & $\left( \begin{matrix} x_1 & 0 \\ 0 & \pm x_1 \end{matrix} \right)$ & $\pm x_1^2$ & $1$ \\
   6 & $\left( \begin{matrix} x_1 & x_2 \\ x_2 & - x_1 \end{matrix} \right)$ & $-x_1^2-x_2^2$ & $2$ \\
\hline
  \end{tabular}
  \caption{List of representatives of $j^1 A$ relative to $\mathcal{G}^1$ along with $j^2 \det A$ and $\textnormal{rank} \; dA \left( 0 \right)$.}
  \label{table:1}
 \end{center}
\end{table}
The
\begin{math}
\mathcal{K}_+^2
\end{math}-type of
\begin{math}
j^2 \det A
\end{math}
and the rank of
\begin{math}
dA
\end{math}
at the origin are invariant under
\begin{math}
\mathcal{G}
\end{math}-equivalence and thus the $6$ classes in Table~\ref{table:1} are distinct classes relative to
\begin{math}
\mathcal{G}^1
\end{math}-equivalence. 
Here, we say that two function germs $f,g:(\mathbb{R}^n,0)\longrightarrow (\mathbb{R},0)$ are
{\it $\mathcal{K}_+$-equivalent} if there exist a diffeomorphism germ $\phi :(\mathbb{R}^n,0)\longrightarrow
(\mathbb{R}^n,0)$ and a function germ $\lambda :(\mathbb{R}^n,0)\longrightarrow  \mathbb{R}$ with
$\lambda (0)>0$ such that $ f\circ \phi (x)=\lambda (x)g(x)$ for any $x\in (\mathbb{R},0)$. Next, we investigate the higher jets for each class by using the complete transversal theorem \cite{Bruce1997} up to
\begin{math}
\mathcal{G}_e
\end{math}-codimension
\begin{math}
8
\end{math}
of corank at most $1$. Since map germs in Class~1 has corank $2$, we investigate the other $5$ cases in what follows. The results are summarized in Table~\ref{table:2}, where BN means the numbers of the normal forms in \textbf{Theorem~1.1 (4)} in \cite{B} (see Table \ref{tableB1}).
\begin{table}[h]
 \begin{center}
  \begin{tabular}{|l|l|l|l|l|l|} \hline
   $\#$ & normal form & range & $\mathcal{G}_e$-cod & corank & BN \\ \hline
   1 & $\left( \begin{matrix} x_1 & 0 \\ 0 & x_2 \end{matrix} \right)$ & & $1$ & $0$ & \\
   2 & $\left( \begin{matrix} x_1 & x_2 \\ x_2 & -x_1 \end{matrix} \right)$ & & $1$ & $0$ & \\
   3 & $\left( \begin{matrix} x_1 & x_2 \\ x_2 & \pm x_1^\ell \end{matrix} \right)$ & $\ell \ge 2$ & $\ell$ & $0$ & $1$ \\
   4 & $\left( \begin{matrix} x_1 & x_2^\ell \\ x_2^\ell & x_1 \end{matrix} \right)$ & $\ell \ge 2$& $2\ell-1$ & $1$ & $1$ \\
   5 & $\left( \begin{matrix} \pm x_2^{\ell_1} & x_1 \\ x_1 & \pm x_2^{\ell_2} \end{matrix} \right)$ & $\ell_1 \ge \ell_2 \ge 2$& $\ell_1 + \ell_2 -1$ & $1$ & $1$ \\
   6 & $\left( \begin{matrix} x_1 & x_2^2 \\ x_2^2 & \pm x_1^2 \end{matrix} \right)$ & & $6$ & $1$ & $5$ \\
   7 & $\left( \begin{matrix} x_1 & 0 \\ 0 & \pm x_2^2 \pm x_1^\ell \end{matrix} \right)$ & $\ell \ge 2$ & $\ell+2$ & $1$ & $2$ \\
   8 & $\left( \begin{matrix} x_1 & 0 \\ 0 & x_1 x_2 \pm x_1^\ell \end{matrix} \right)$ & $\ell \ge 3$ & $2 \ell$ & $1$ & $3$ \\
   9 & $\left( \begin{matrix} x_1 & x_2^\ell \\ x_2^\ell & x_1 x_2 \end{matrix} \right)$ & $\ell \ge 2$ & $2 \ell + 1$ & $1$ & $4$ \\
   10 & $\left( \begin{matrix} x_1 & 0 \\ 0 & \pm x_1^2 + x_2^3 \end{matrix} \right)$ & & $7$ & $1$ & $6$ \\
\hline
  \end{tabular}
  \caption{ $\mathcal{G}$-singularities  $(\mathbb R^2,0)\rightarrow \textnormal{Sym}_2$  of $\mathcal{G}_e$-codimension less than $9$ and corank at most $1$ where one of the $\pm$ coincides if $\ell_1$ or $\ell_2$ or both odd in the class 5. BN is a corresponding number in Table \ref{tableB1}. }
  \label{table:2}
 \end{center}
\end{table}

\begin{table}[h]
 \begin{center}
  \begin{tabular}{|l|l|l|l|l|} \hline
   $\#$ & normal form & range & $\mathcal{G}_e$-cod & BN \\ \hline
   1 & $\left( \begin{matrix} x_1 & 0 & 0 \\ 0 & x_2 & 0 \\ 0 & 0 & x_1 - x_2 \end{matrix} \right)$ & & $4$ & $1$ \\
   2 & $\left( \begin{matrix} x_1 & 0 & 0 \\ 0 & x_2 & 0 \\ 0 & 0 & \pm \left( x_1 + x_2 \right) \end{matrix} \right)$ & & $4$ & $1$ \\
   3 & $\left( \begin{matrix} \pm x_1 & 0 & 0 \\ 0 & x_1 & x_2 \\ 0 & x_2 & -x_1 \end{matrix} \right)$ & & $4$ & $1$ \\
   4 & $\left( \begin{matrix} \pm x_2^{\ell_1} & x_1 & 0 \\ x_1 & \pm x_2^{\ell_2} & 0 \\ 0 & 0 & x_2 \end{matrix} \right)$ & $\begin{matrix} \ell_1 \ge \ell_2 \ge 1 \\ \ell_1 \ge 2 \end{matrix}$ & $\ell_1 + \ell_2 + 2$ & $1$ \\
   5 & $\left( \begin{matrix} x_1 & x_2^\ell & 0 \\ x_2^\ell & x_1 & 0 \\ 0 & 0 & x_2 \end{matrix} \right)$ & $\ell \ge 2$& $2\ell+2$ & $1$ \\
   6 & $\left( \begin{matrix} 0 & x_2 & x_1 \\ x_2 & x_1 & 0 \\ x_1 & 0 & \pm x_2^2 \end{matrix} \right)$ & & $6$ & $2$ \\
   7 & $\left( \begin{matrix} 0 & x_2 & x_1 \\ x_2 & x_1 & 0 \\ x_1 & 0 & x_1 x_2 \end{matrix} \right)$ & & $7$ & $3$ \\
   8 & $\left( \begin{matrix} x_2 & x_1 & 0 \\ x_1 & \pm x_2^2 & 0 \\ 0 & 0 & x_1 \end{matrix} \right)$ & & $7$ & $5$ \\
   9 & $\left( \begin{matrix} 0 & x_2 & x_1 \\ x_2 & x_1 & 0 \\ x_1 & 0 & x_2^3 \end{matrix} \right)$ & & $8$ & $4$ \\
   10 & $\left( \begin{matrix} x_2 & x_1 & 0 \\ x_1 & 0 & x_2^2 \\ 0 & x_2^2 & x_1 \end{matrix} \right)$ & & $8$ & $6$ \\
\hline
  \end{tabular}
  \caption{ $\mathcal{G}$-singularities  $(\mathbb R^2,0)\rightarrow \textnormal{Sym}_3$ of $\mathcal{G}_e$-codimension less than $9$. In this case, the corank of all the singularities is $0$ at the origin. BN is a corresponding number in Table \ref{tableB2}.}
  \label{table:n2m3}
 \end{center}
\end{table}
We illustrate our classification procedure in case of Class~2. The other $4$ classes can be handled similarly and we omit the proof. In case of Class~2, any representative of $2$-jets relative to
\begin{math}
\mathcal{G}^2
\end{math}
can be written as
\begin{equation}
j^2 A_c = \left(
\begin{matrix}
x_1 & c_1 x_2^2 \\
c_1 x_2^2 & c_2 x_2^2 + c_3 x_1 x_2 + c_4 x_1^2
\end{matrix} \right)
\end{equation}
by using the complete transversal theorem \cite{Bruce1997}, where
\begin{math}
c_1, c_2, c_3, c_4 \in \mathbb{R}
\end{math}. In what follows, we write
\begin{math}
c = \left( c_1, \cdots, c_4 \right)
\end{math}. Next, we use Mather's lemma \cite{Mather1969} to normalize the $2$-jet. First, let us decompose the $4$-dimensional parameter space into semi-algebraic sets such that the dimension of 
\begin{math}
 T_{j^2 A_c} \mathcal{G}^2 \cdot j^2 A_c
\end{math}
is constant in each semi-algebraic set, which is one of the conditions for the set of the corresponding
\begin{math}
j^2 A_c
\end{math}
is in a single
\begin{math}
\mathcal{G}^2
\end{math}-orbit. The result is summarized in Table~\ref{table:gcodim_c2}.
\begin{table}[h]
 \begin{center}
  \begin{tabular}{|l|l|l|} \hline
   class number & $\dim_{\mathbb{R}} T_{j^2 A_c} \mathcal{G}^2 \cdot j^2 A_c$ & semi-algebraic set \\ \hline
   2-1 & $11$ & $c_2 \left( - c_3^2 + 4 c_2 c_4 \right) \neq 0$ \\
   2-2 & $10$ & $- c_3^2 + 4 c_2 c_4 = 0, c_2 \neq 0$ \\
         &        & or $c_2 = 0, c_1 c_3 \neq 0$ \\
   2-3 & $9$  & $c_1 = c_2 = 0, c_3 \neq 0$ \\
         &    & or $c_2 = c_3 = 0, c_1 c_4 \neq 0$ \\
   2-4 & $8$ & $c_1 = c_2 = c_3 = 0, c_4 \neq 0$ \\
         &    & or $c_2 = c_3 = c_4 = 0, c_1 \neq 0$ \\
   2-5 & $7$ & $c_1 = c_2 = c_3 = c_4 = 0$ \\
\hline
  \end{tabular}
  \caption{Decomposition of $\mathbb{R}^4$ into semi-algebraic sets on which $\dim_{\mathbb{R}} T_{j^2 A_c} \mathcal{G}^2 \cdot j^2 A_c$ is constant.}
  \label{table:gcodim_c2}
 \end{center}
\end{table}
Each semi-algebraic set in Table~\ref{table:gcodim_c2} is
\begin{math}
C^\infty
\end{math}
manifold and its tangent space is contained in
\begin{math}
T_{j^2 A_c} \mathcal{G}^2 \cdot j^2 A_c
\end{math}. Therefore, Mather's lemma implies that each connected component of the semi-algebraic sets is contained in a single
\begin{math}
\mathcal{G}^2
\end{math}-orbit. In what follows, we pick up a representative for each connected component of the semi-algebraic sets in Table~\ref{table:gcodim_c2}.
\subsection{Class~2-1}
The semi-algebraic set defined by
\begin{math}
c_2 \left( - c_3^2 + 4 c_2 c_4 \right) \neq 0
\end{math}
consists of
\begin{math}
4
\end{math}
connected components and we can pick up representatives
\begin{math}
c = \left( 0, \pm 1, 0, \pm 1 \right)
\end{math}
from the $4$ connected components. The corresponding $2$-jets are
\begin{math}
\left(
\begin{matrix}
x_1 & 0 \\
0 & \pm x_2^2 \pm x_1^2
\end{matrix}
\right)
\end{math}.
\subsection{Class~2-2}
The semi-algebraic set defined by
\begin{math}
- c_3^2 + 4 c_2 c_4 = 0, c_2 \neq 0
\end{math}
or
\begin{math}
c_2 = 0, c_1 c_3 \neq 0
\end{math}
consists of
\begin{math}
6
\end{math}
connected components and we can pick up representatives
\begin{math}
c = \left( 0, \pm 1, 0, 0 \right)
\end{math}
and
\begin{math}
c = \left( \pm 1, 0, \pm 1, 0 \right)
\end{math}
from the $6$ connected components. The corresponding $2$-jets are
\begin{math}
\left(
\begin{matrix}
x_1 & 0 \\
0 & \pm x_2^2
\end{matrix}
\right)
\end{math}
and
\begin{math}
\left(
\begin{matrix}
x_1 & \pm x_2^2 \\
\pm x_2^2 & \pm x_1 x_2
\end{matrix}
\right)
\end{math}. In the latter case, if the sign in front of
\begin{math}
x_2^2
\end{math}
is negative, we can multiply
\begin{math}
\left(
\begin{matrix}
-1 & 0 \\
0 & 1
\end{matrix}
\right)
\end{math}
from the both side to get
\begin{math}
\left(
\begin{matrix}
x_1 & x_2^2 \\
x_2^2 & \pm x_1 x_2
\end{matrix}
\right)
\end{math}. By changing the sign of
\begin{math}
x_2
\end{math}
if necessary, we get
\begin{math}
\left(
\begin{matrix}
x_1 & x_2^2 \\
x_2^2 & x_1 x_2
\end{matrix}
\right)
\end{math}.
\subsection{Class~2-3}
In the same manner as above, we get the representatives
\begin{math}
\left(
\begin{matrix}
x_1 & 0 \\
0 & x_1 x_2
\end{matrix}
\right)
\end{math}
and
\begin{math}
\left(
\begin{matrix}
x_1 & x_2^2 \\
x_2^2 & x_1^2
\end{matrix}
\right)
\end{math}.
\subsection{Class~2-4}
In the same manner as above, we get the representatives
\begin{math}
\left(
\begin{matrix}
x_1 & 0 \\
0 & x_1^2
\end{matrix}
\right)
\end{math}
and
\begin{math}
\left(
\begin{matrix}
x_1 & x_2^2 \\
x_2^2 & 0
\end{matrix}
\right)
\end{math}.
\subsection{Class~2-5}
In this case, all the coefficients are zero and the $2$-jet is
\begin{math}
\left(
\begin{matrix}
x_1 & 0 \\
0 & 0
\end{matrix}
\right)
\end{math}.

The results are summarized in Table~\ref{table:c22jets}.
\begin{table}[h]
 \begin{center}
  \begin{tabular}{|l|l|l|} \hline
   class num & representative & $\mathcal{G}$-codim \\ \hline
   2-1-1 & $\left( \begin{matrix} x_1 & 0 \\ 0 & \pm x_1^2 \pm x_1^2 \end{matrix} \right)$ & $11$ \\
   2-2-1 & $\left( \begin{matrix} x_1 & 0 \\ 0 & \pm x_2^2 \end{matrix} \right)$ & $10$ \\
   2-2-2 & $\left( \begin{matrix} x_1 & x_2^2 \\ x_2^2 & x_1 x_2 \end{matrix} \right)$ & $10$ \\
   2-3-1 & $\left( \begin{matrix} x_1 & 0 \\ 0 & x_1 x_2 \end{matrix} \right)$ & $9$ \\
   2-3-2 & $\left( \begin{matrix} x_1 & x_2^2 \\ x_2^2 & 0 \end{matrix} \right)$ & $9$ \\
   2-4-1 & $\left( \begin{matrix} x_1 & 0 \\ 0 & \pm x_1^2 \end{matrix} \right)$ & $8$ \\
   2-4-2 & $\left( \begin{matrix} x_1 & x_2^2 \\ x_2^2 & 0 \end{matrix} \right)$ & $8$ \\
   2-5 & $\left( \begin{matrix} x_1 & 0 \\ 0 & 0 \end{matrix} \right)$ & $7$ \\
\hline
  \end{tabular}
  \caption{Representatives of $2$-jets in Class~2.}
  \label{table:c22jets}
 \end{center}
\end{table}
We continue classification of higher jets for each case and stops either a given $k$-jet is $k$-determined or the lower bound of the
\begin{math}
\mathcal{G}_e
\end{math}-codimension of map-germs having a given $k$-jet is larger than the prescribed value of
\begin{math}
\mathcal{G}_e
\end{math}-codimension ($8$ in the current situation). For the detail of the estimation, see \cite{Wall1981}.
\subsection{Key Theorem for Real Classification}
In this section, we would like to introduce a key theorem for classification in real. This theorem provides a useful criterion to check if two map germs of the same complex class, like,
\begin{equation}
A_{\pm} \left( x \right) = \left(
\begin{matrix}
x_1 & 0 \\
0 & x_1 x_2 \pm x_2^k
\end{matrix}
\right),
\end{equation}
belong to the same real class or not. Let
\begin{math}
A \colon \left( \mathbb{R}^r, 0 \right) \rightarrow \left( \textnormal{Sym}_n, O_n \right)
\end{math}
and
\begin{math}
S^{\left( n_1, n_2, n_3 \right)} \left( A \right)
\end{math}
be subset germ of
\begin{math}
\left( \mathbb{R}^r, 0 \right)
\end{math}
such that
\begin{math}
A \left( x \right)
\end{math}
has
\begin{math}
n_1
\end{math}
positive eigenvalues,
\begin{math}
n_2
\end{math}
zero eigenvalues, and
\begin{math}
n_3
\end{math}
negative eigenvalues for all
\begin{math}
x \in S^{\left( n_1, n_2, n_3 \right)} \left( A \right)
\end{math}. Since the action of
\begin{math}
\mathcal{H}
\end{math}
to
\begin{math}
A
\end{math}
preserves
\begin{math}
S^{\left( n_1, n_2, n_3 \right)} \left( A \right)
\end{math}
for
\begin{math}
n_1, n_2, n_3 \in \mathbb{N}
\end{math}, we get the following theorem.
\begin{thm}
If
\begin{math}
A, B \colon \left( \mathbb{R}^r, 0 \right) \rightarrow \left( \textnormal{Sym}_n, O_n \right)
\end{math}
are
\begin{math}
\mathcal{G}
\end{math}-equivalent, then 
\begin{math}
S^{\left( n_1, n_2, n_3 \right)} \left( A \right)
\end{math}
and
\begin{math}
S^{\left( n_1, n_2, n_3 \right)} \left( B \right)
\end{math}
are diffeomorphic for all
\begin{math}
n_1, n_2, n_3 \in \mathbb{N}
\end{math}
\end{thm}
Let us demonstrate the theorem to check if the two map germs
\begin{math}
A_{\pm}
\end{math}
are in the same
\begin{math}
\mathcal{G}
\end{math}-equivalence class in real or not. In Fig.~\ref{fig:pos_neq}, we plot
\begin{math}
S^{\left( 2, 0, 0 \right)} \left( A_{\pm} \right)
\end{math}
for
\begin{math}
k = 2
\end{math}
in
\begin{math}
\mathbb{R}^2
\end{math}
indicated in blue.
\begin{figure}
  \centering
  \includegraphics[width=5cm]{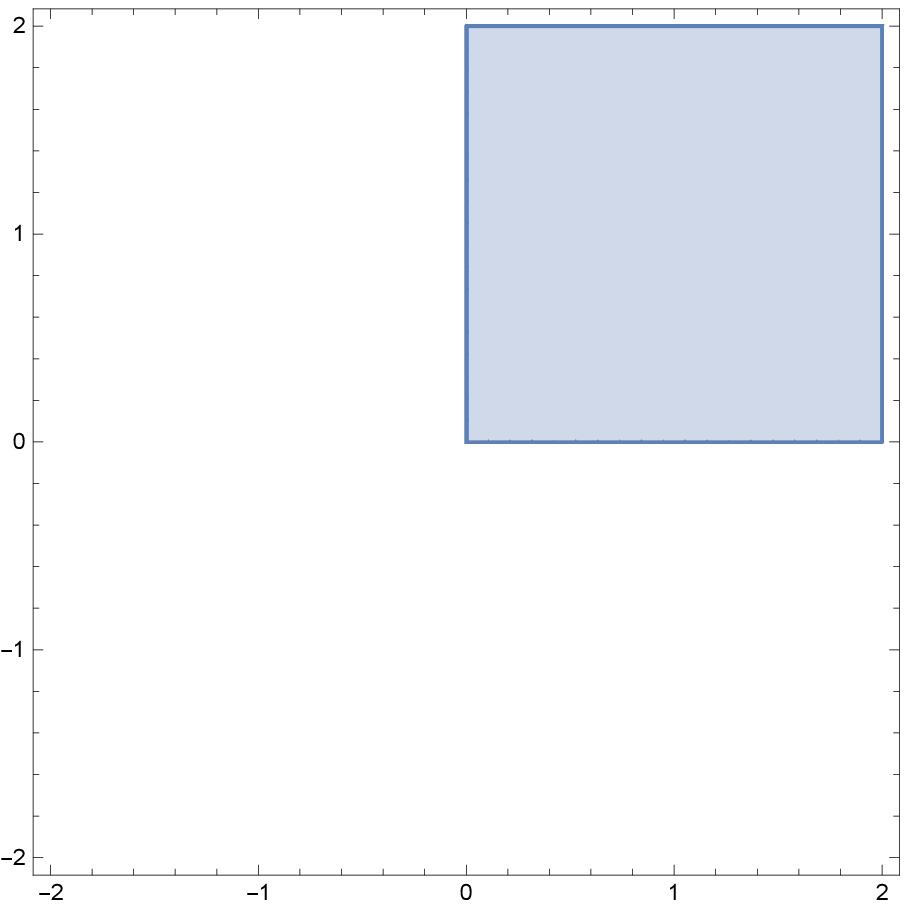}
  \includegraphics[width=5cm]{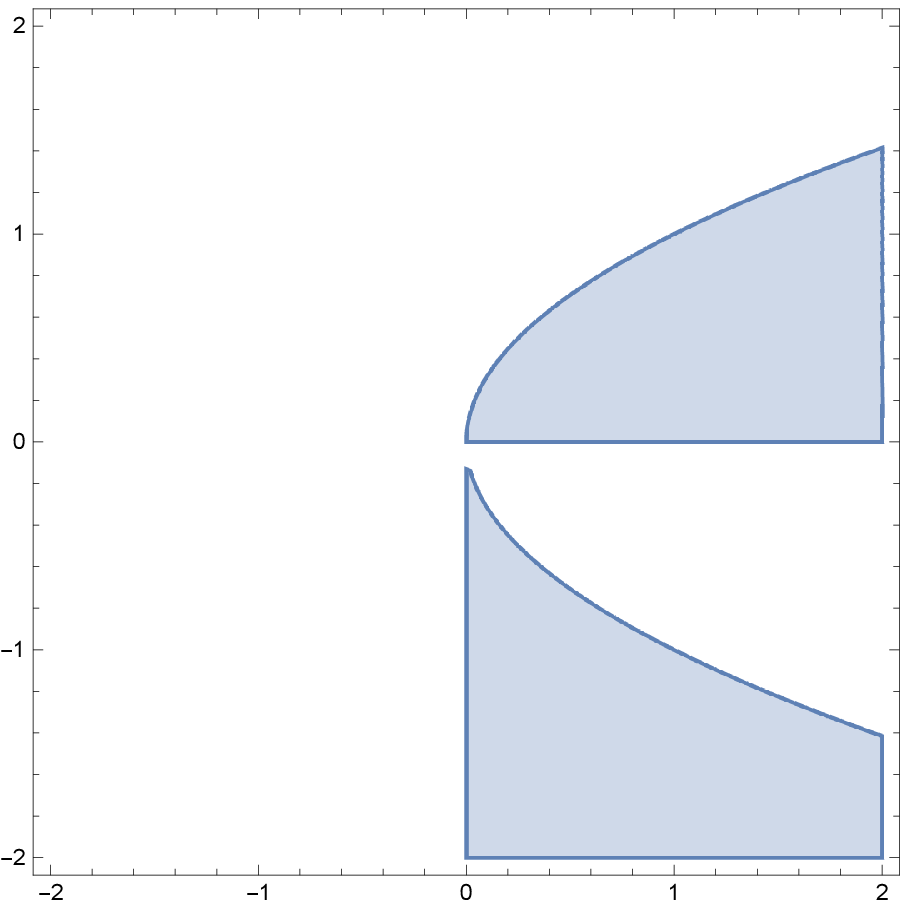}
  \caption{$S^{\left( 2, 0, 0 \right)} \left( A_+ \right)$ (left) and $S^{\left( 2, 0, 0 \right)} \left( A_- \right)$ (right) in $\mathbb{R}^2$ where the horizontal axis is $x_1$ and the vertical axis is $x_2$.}
  \label{fig:pos_neq}
\end{figure}
Since
\begin{math}
S^{\left( 2, 0, 0 \right)} \left( A_+ \right)
\end{math}
consists of single connected region whereas
\begin{math}
S^{\left( 2, 0, 0 \right)} \left( A_- \right)
\end{math}
consists of two separated regions, they are not diffeomorphic. Therefore, we can conclude that the two map germs
\begin{math}
A_+
\end{math}
and
\begin{math}
A_-
\end{math}
are not
\begin{math}
\mathcal{G}
\end{math}-equivalent in real.

\section{Unimodular normal forms of map-germs $\left( \mathbb{R}^r, 0 \right) \rightarrow \left( \textnormal{Sym}_n, O_n \right)$  for $r = 2, n = 2, 3$} \label{sec:uni_normal_form}
Unimodular normal forms for families of symmetric matrices, which belong to $ C^{\infty}\left( \mathbb{R}^2, \textnormal{Sym}_2 \right)$,  are presented in Table \ref{table:3}. It is easy to check that all real normal forms in Table \ref{table:2} are symmetrically quasi-homogeneous with weights $\lambda_1$, $\lambda_2$ presented in Table \ref{table:3}.  By Proposition \ref{split} we show that only real singularities 8-10 in Table \ref{table:2} spilt into two different unimodular singularities.

\begin{table}[h]
 \begin{center}
  \begin{tabular}{|l|l|l|l|l|l|} \hline
   $\#$ & normal form & range & $\mathcal{G}_e$-cod &  $\lambda_1$ & $\lambda_2$ \\ \hline
   1 & $\left( \begin{matrix} x_1 & 0 \\ 0 & x_2 \end{matrix} \right)$ & & $1$ & $1$ & $1$\\
   2 & $\left( \begin{matrix} x_1 & x_2 \\ x_2 & -x_1 \end{matrix} \right)$ & & $1$ & $1$ &  $1$\\
   3 & $\left( \begin{matrix} x_1 & x_2 \\ x_2 & \pm x_1^\ell \end{matrix} \right)$ & $\ell \ge 2$ & $\ell$ & $2$ & $\ell+1$ \\
   4 & $\left( \begin{matrix} x_1 & x_2^\ell \\ x_2^\ell & x_1 \end{matrix} \right)$ & $\ell \ge 2$& $2\ell-1$ & $\ell$ & $1$ \\
   5 & $\left( \begin{matrix} \pm x_2^{\ell_1} & x_1 \\ x_1 & \pm x_2^{\ell_2} \end{matrix} \right)$ & $\ell_1 \ge \ell_2 \ge 2$& $\ell_1 + \ell_2 -1$ & $\ell_1+\ell_2$ & $2$ \\
   6 & $\left( \begin{matrix} x_1 & x_2^2 \\ x_2^2 & \pm x_1^2 \end{matrix} \right)$ & & $6$ & $4$ & $3$ \\
   7 & $\left( \begin{matrix} x_1 & 0 \\ 0 & \pm x_2^2 \pm x_1^\ell \end{matrix} \right)$ & $\ell \ge 2$ & $\ell+2$ & $2$ & $\ell$ \\
   $8^{\pm}$ & $\left( \begin{matrix} x_1 & 0 \\ 0 &\pm x_1 x_2 \pm x_1^\ell \end{matrix} \right)$ & $\ell \ge 3$ & $2 \ell$ & $1$ & $\ell - 1$ \\
   $9^{+}$ & $\left( \begin{matrix} x_1 & x_2^\ell \\ x_2^\ell & x_1 x_2 \end{matrix} \right)$ & $\ell \ge 2$ & $2 \ell + 1$ & $2\ell-1$ & $2$ \\
   $9^{-}$ & $\left( \begin{matrix} -x_1 & x_2^\ell \\ x_2^\ell & -x_1 x_2 \end{matrix} \right)$ & $\ell \ge 2$ & $2 \ell + 1$ & $2\ell-1$ & $2$ \\
   $10^{\pm}$ & $\left( \begin{matrix} x_1 & 0 \\ 0 & \pm x_1^2 \pm x_2^3 \end{matrix} \right)$ & & $7$ & $3$ & $2$ \\
\hline
  \end{tabular}
  \caption{Unimodular singularities $(\mathbb R^2,0)\rightarrow \textnormal{Sym}_2$. $\lambda_1, \lambda_2$ are weigths  of symmetrical quasi-homogeneity. }
  \label{table:3}
 \end{center}
\end{table}

 Let $\Omega=dx_1\wedge dx_2$ be a volume form-germ on $\mathbb R^2$. In Table \ref{table:orient-revers} we present orientation-reversing diffeomorphism-germs  $\Phi:(\mathbb R^2,0)\rightarrow (\mathbb R^2,0)$ that  belong to $\mathcal D_A$ for the following  $A:\left( \mathbb{R}^2, 0 \right) \rightarrow \left( \textnormal{Sym}_2, O_2 \right)$  i.e. $\Phi^{\ast}\Omega=-\Omega$ and $A=X^T(A\circ \Phi)X$ for the following $X\in C^{\infty}(\mathbb R^2,\textnormal{GL}_2)$.

\begin{table}[h]
 \begin{center}
  \begin{tabular}{|l|l|l|l|} \hline
   $\#$ &  $A(x_1,x_2)$ & $\Phi(x_1,x_2)$ & $X(x_1,x_2)$  \\ \hline
   1 & $\left( \begin{matrix} x_1 & 0 \\ 0 & x_2 \end{matrix} \right)$ &$(x_2,x_1)$ & $\left( \begin{matrix} 0 & 1 \\ 1 & 0 \end{matrix} \right)$ \\
   2 & $\left( \begin{matrix} x_1 & x_2 \\ x_2 & -x_1 \end{matrix} \right)$ & $(-x_1,x_2)$ & $\left( \begin{matrix} 0 & 1 \\ 1 & 0 \end{matrix} \right)$ \\
   3 & $\left( \begin{matrix} x_1 & x_2 \\ x_2 & \pm x_1^\ell \end{matrix} \right)$ & $(x_1,-x_2)$ & $\left( \begin{matrix} -1 & 0 \\ 0 & 1 \end{matrix} \right)$  \\
   4 & $\left( \begin{matrix} x_1 & x_2^\ell \\ x_2^\ell & x_1 \end{matrix} \right)$ & $(x_1,-x_2)$& $\left( \begin{matrix} (-1)^{\ell} & 0 \\ 0 & 1 \end{matrix} \right)$  \\
   5 & $\left( \begin{matrix} \pm x_2^{\ell_1} & x_1 \\ x_1 & \pm x_2^{\ell_2} \end{matrix} \right)$ & $(-x_1,x_2)$& $\left( \begin{matrix} -1 & 0 \\ 0 & 1 \end{matrix} \right)$  \\
   6 & $\left( \begin{matrix} x_1 & x_2^2 \\ x_2^2 & \pm x_1^2 \end{matrix} \right)$ &$(x_1,-x_2)$ & $\left( \begin{matrix} 1 & 0 \\ 0 & 1 \end{matrix} \right)$   \\
   7 & $\left( \begin{matrix} x_1 & 0 \\ 0 & \pm x_2^2 \pm x_1^\ell \end{matrix} \right)$ & $(x_1,-x_2$ & $\left( \begin{matrix} 1 & 0 \\ 0 & 1 \end{matrix} \right)$  \\
\hline
  \end{tabular}
  \caption{Orientation-reversing diffeomorphism-germs $\Phi \in \mathcal D_A$ for$A:\left( \mathbb{R}^2, 0 \right) \rightarrow \left( \textnormal{Sym}_2, O_2 \right)$.}
  \label{table:orient-revers}
 \end{center}
\end{table}

Now we show that if $\Phi \in \mathcal D_A$ then $\Phi$ preserves the orientation of $\mathbb R^2$ for $A$ in rows 8-10 in Table \ref{table:2}.
Let $A= \left( \begin{matrix} x_1 & 0 \\ 0 & x_1 x_2 \pm x_1^\ell \end{matrix} \right)$ and  $\Phi=(\Phi_1,\Phi_2):(\mathbb R^2,0)\rightarrow (\mathbb R^2,0)$ be a diffeomorphism-germ such that
\begin{equation}\label{XATA}
B:=A\circ \Phi-X^T AX=0
\end{equation}
 for some  $X\in C^{\infty}(\mathbb R^2,\textnormal{GL}_2)$. It is easy to see that  $\frac{\partial B_{11}}{\partial x_2}|_0=\frac{\partial \Phi_1}{\partial x_2}|_0=0$ and $\frac{\partial B_{22}}{\partial x_1}|_0=-(X_{1,2}(0))^2=0$. It implies that  $\frac{\partial^2B_{22}}{\partial x_1 \partial x_2}|_0=\frac{\partial \Phi_1}{\partial x_1}|_0\frac{\partial \Phi_2}{\partial x_2}|_0-(X_{22}(0))^2=0$. Thus $\det d\Phi|_0$ is positive.  In the same way we prove that any $\Phi\in \mathcal D_A$ for $A=\left( \begin{matrix} x_1 & x_2^\ell \\ x_2^\ell & x_1 x_2 \end{matrix} \right)$ preserves the orientation. For $A=\left( \begin{matrix} x_1 & 0 \\ 0 & \pm x_1^2 + x_2^3 \end{matrix} \right)$ we proceed in the similar way. By $\frac{\partial B_{11}}{\partial x_1}|_0=0$, $\frac{\partial B_{11}}{\partial x_2}|_0=0$ and $\frac{\partial^3 B_{22}}{\partial x_2^3}|_0=0$ we obtain that  $\det d\Phi|_0$ is positive.

Unimodular normal forms for families of symmetric matrices, which belong to  $ C^{\infty}\left( \mathbb{R}^2, \textnormal{Sym}_3, O_3 \right)$, are presented in Table \ref{table:Uni-n2m3}. It is easy to check that all real normal forms in Table \ref{table:n2m3} are symmetrically quasi-homogeneous with weights $\lambda_1$, $\lambda_2$ presented in Table \ref{table:Uni-n2m3}.  By Proposition \ref{split} we show that only real singularities 5 and 7-10 in Table \ref{table:n2m3} split into two different unimodular singularities. Let $\Omega=dx_1\wedge dx_2$ be a volume form-germ on $\mathbb R^2$.  In Table \ref{table:orient-revers2} we present orientation-reversing diffeomorphism-germs  $\Phi:(\mathbb R^2,0)\rightarrow (\mathbb R^2,0)$ that  belong to $\mathcal D_A$ for the following  $A:\left( \mathbb{R}^2, 0 \right) \rightarrow \left( \textnormal{Sym}_3, O_3 \right)$  i.e. $\Phi^{\ast}\Omega=-\Omega$ and $A=X^T(A\circ \Phi)X$ for the following $X\in C^{\infty}(\mathbb R^2,\textnormal{GL}_3)$.

\begin{table}[h]
 \begin{center}
  \begin{tabular}{|l|l|l|l|} \hline
   $\#$ &  $A(x_1,x_2)$ & $\Phi(x_1,x_2)$ & $X(x_1,x_2)$  \\ \hline
   1 & $\left( \begin{matrix} x_1 & 0 & 0 \\ 0 & x_2 & 0 \\ 0 & 0 & x_1 - x_2 \end{matrix} \right)$ &$(x_1,x_1-x_2)$ & $\left( \begin{matrix} 1 & 0 & 0 \\ 0 & 0 & 1 \\ 0 &1 & 0 \end{matrix} \right)$ \\
   2 & $\left( \begin{matrix} x_1 & 0 & 0 \\ 0 & x_2 & 0 \\ 0 & 0 & \pm \left( x_1 + x_2 \right) \end{matrix} \right)$ & $(x_2,x_1)$ & $\left( \begin{matrix} 0 & 1 & 0 \\ 1 & 0 & 0 \\ 0 & 0 & 1 \end{matrix} \right)$ \\
   3 & $\left( \begin{matrix} \pm x_1 & 0 & 0 \\ 0 & x_1 & x_2 \\ 0 & x_2 & -x_1 \end{matrix} \right)$ & $(x_1,-x_2)$ & $\left( \begin{matrix} 1 & 0 & 0 \\ 0 & -1 & 0 \\ 0 &0 & 1 \end{matrix} \right)$  \\
   4 & $\left( \begin{matrix} \pm x_2^{\ell_1} & x_1 & 0 \\ x_1 & \pm x_2^{\ell_2} & 0 \\ 0 & 0 & x_2 \end{matrix} \right) $ & $(-x_1,x_2)$& $\left( \begin{matrix} 1 & 0 & 0 \\ 0 & -1 & 0 \\ 0 &0 & 1 \end{matrix} \right)$  \\
   6 & $\left( \begin{matrix} 0 & x_2 & x_1 \\ x_2 & x_1 & 0 \\ x_1 & 0 & \pm x_2^2 \end{matrix} \right)$ & $(x_1,-x_2)$& $\left( \begin{matrix} 1 & 0 & 0 \\ 0 & -1 & 0 \\ 0 &0 & 1 \end{matrix} \right)$  \\

\hline
  \end{tabular}
  \caption{Orientation-reversing diffeomorphism-germs $\Phi \in \mathcal D_A$ for $A:\left( \mathbb{R}^2, 0 \right) \rightarrow \left( \textnormal{Sym}_3, O_3 \right)$. }
  \label{table:orient-revers2}
 \end{center}
\end{table}

Now we show that  if $\Phi \in \mathcal D_A$ then $\Phi$ preserves the orientation of $\mathbb R^2$ for  $A=\left( \begin{matrix} x_1 & x_2^\ell & 0 \\ x_2^\ell & x_1 & 0 \\ 0 & 0 & x_2 \end{matrix} \right)$ (see row 5 in Table \ref{table:n2m3}).
Let  $\Phi=(\Phi_1,\Phi_2):(\mathbb R^2,0)\rightarrow (\mathbb R^2,0)$ be a diffeomorphism-germ such that
\begin{equation}\label{XATA3}
C:=A\circ \Phi-X^T AX=0
\end{equation}
 for some  $X\in C^{\infty}(\mathbb R^2,\textnormal{GL}_2)$.  From (\ref{XATA3}) it is easy to see that
\begin{equation}\label{eq7}
\frac{\partial C_{12}}{\partial x_2}|_0=-X_{31}(0)X_{32}(0)=0,
\end{equation}
\begin{equation}\label{eq6}
\frac{\partial C_{13}}{\partial x_2}|_0=-X_{31}(0)X_{33}(0)=0,
\end{equation}
\begin{equation}\label{eq5}
\frac{\partial C_{23}}{\partial x_2}|_0=-X_{32}(0)X_{33}(0)=0,
\end{equation}
\begin{equation}\label{eq4}
\frac{\partial C_{11}}{\partial x_2}|_0=\frac{\partial \Phi_1}{\partial x_2}|_0-(X_{31}(0))^2=0,
\end{equation}
\begin{equation}\label{eq3}
\frac{\partial C_{33}}{\partial x_2}|_0=\frac{\partial \Phi_2}{\partial x_2}|_0-(X_{33}(0))^2=0,
\end{equation}
\begin{equation}\label{eq2}
\frac{\partial C_{22}}{\partial x_2}|_0=\frac{\partial \Phi_1}{\partial x_2}|_0-(X_{32}(0))^2=0,
\end{equation}
\begin{equation}\label{eq1}
\frac{\partial C_{11}}{\partial x_1}|_0=\frac{\partial \Phi_1}{\partial x_1}|_0-(X_{11}(0))^2-(X_{21}(0))^2=0.
\end{equation}
(\ref{eq7})-(\ref{eq5}) imply that $X_{31}(0)=X_{32}(0)=0$ or $X_{31}(0)=X_{33}(0)=0$ or $X_{32}(0)=X_{33}(0)=0$.
But if $X_{31}(0)=X_{33}(0)=0$ or $X_{32}(0)=X_{33}(0)=0$ then by (\ref{eq4})-(\ref{eq2}) we obtain that $\det d\Phi|_0=0$, which is impossible. If  $X_{31}(0)=X_{32}(0)=0$ then by (\ref{eq4})-(\ref{eq1}) we have that $\det d\Phi|_0=((X_{11}(0))^2+(X_{21}(0))^2)(X_{33}(0))^2$. Thus $\det d\Phi|_0$ is positive.  In the similar way one can  prove that any $\Phi\in \mathcal D_A$  preserves the orientation for $A$ presented in rows 7-10 in Table \ref{table:n2m3}.
\begin{table}[h]
 \begin{center}
  \begin{tabular}{|l|l|l|l|l|l|} \hline
   $\#$ & representative & range & $\mathcal{G}_e$-cod & $\lambda_1$ & $\lambda_2$ \\ \hline
   1 & $\left( \begin{matrix} x_1 & 0 & 0 \\ 0 & x_2 & 0 \\ 0 & 0 & x_1 - x_2 \end{matrix} \right)$ & & $4$ & $1$ & $1$ \\
   2 & $\left( \begin{matrix} x_1 & 0 & 0 \\ 0 & x_2 & 0 \\ 0 & 0 & \pm \left( x_1 + x_2 \right) \end{matrix} \right)$ & & $4$ & $1$ & $1$ \\
   3 & $\left( \begin{matrix} \pm x_1 & 0 & 0 \\ 0 & x_1 & x_2 \\ 0 & x_2 & -x_1 \end{matrix} \right)$ & & $4$ & $1$ & $1$ \\
   4 & $\left( \begin{matrix} \pm x_2^{\ell_1} & x_1 & 0 \\ x_1 & \pm x_2^{\ell_2} & 0 \\ 0 & 0 & x_2 \end{matrix} \right)$ & $\begin{matrix} \ell_1 \ge \ell_2 \ge 1 \\ \ell_1 \ge 2 \end{matrix}$ & $\ell_1 + \ell_2 + 2$ & $\ell_1 + \ell_2$ & $2$ \\
  $5^{\pm}$ & $\left( \begin{matrix} x_1 & x_2^\ell & 0 \\ x_2^\ell & x_1 & 0 \\ 0 & 0 &\pm x_2 \end{matrix} \right)$ & $\ell \ge 2$& $2\ell+2$ & $\ell$ & $1$ \\
   6 & $\left( \begin{matrix} 0 & x_2 & x_1 \\ x_2 & x_1 & 0 \\ x_1 & 0 & \pm x_2^2 \end{matrix} \right)$ & & $6$ & $4$ & $3$ \\
   $7^{\pm}$ & $\left( \begin{matrix} 0 & x_2 & x_1 \\ x_2 &\pm  x_1 & 0 \\ x_1 & 0 & x_1 x_2 \end{matrix} \right)$ & & $7$ & $3$ & $2$ \\
   $8^{\pm}$ & $\left( \begin{matrix} \pm x_2 & x_1 & 0 \\ x_1 & \pm x_2^2 & 0 \\ 0 & 0 & x_1 \end{matrix} \right)$ & & $7$ & $3$ & $2$ \\
   $9^{\pm}$ & $\left( \begin{matrix} 0 & x_2 & x_1 \\ x_2 &\pm x_1 & 0 \\ x_1 & 0 & x_2^3 \end{matrix} \right)$ & & $8$ & $5$ & $3$ \\

   $10^{\pm}$ & $\left( \begin{matrix} \pm x_2 & x_1 & 0 \\ x_1 & 0 & x_2^2 \\ 0 & x_2^2 & x_1 \end{matrix} \right)$ & & $8$ & $5$ & $3$ \\
\hline
  \end{tabular}
  \caption{Unimodular singularities $(\mathbb R^2,0)\rightarrow \textnormal{Sym}_3$. $\lambda_1$, $\lambda_2$ are weigths  of symmetrical quasi-homogeneity.}
  \label{table:Uni-n2m3}
 \end{center}
\end{table}

\section{Examples for non symmetrically quasi-homogeneous map-germs} \label{sec:example}
In this section, we provide some examples of map-germs
\begin{math}
\left( \mathbb{R}^r, 0 \right) \rightarrow \left( \textnormal{Sym}_n, O_n \right)
\end{math}
that are not symmetrically quasi-homogeneous. A natural candidate for such a map-germ is a non-simple map-germ at the boundary of the set of simple map-germs in Table~\ref{table:2} and Table~\ref{table:n2m3}. The most generic non-simple map-germs appearing in Class~2-4-1 in Table~\ref{table:c22jets} are bi-modal map-germs of stratum codimension $8$. It is too complicated to show the whole parameter families and thus we pick up some of them to give illustrative examples.
\begin{equation}
A_{\alpha,\beta} \left( x \right) = \left(
\begin{matrix}
x_1 & x_2^3 \\
x_2^3 & \delta x_1^2 + \alpha x_1 x_2^2 + \beta x_2^4
\end{matrix}
\right)
\end{equation}
where
\begin{math}
\delta = \pm 1
\end{math}
and
\begin{math}
\alpha, \beta \in \mathbb{R}
\end{math}
are moduli parameters satisfying
\begin{math}
\alpha^3 \beta^2 + 4 \alpha^4 - 4 \delta \alpha \beta^3 - 18 \delta \alpha^2 \beta - 27 \alpha \neq 0
\end{math},
\begin{equation}
A_{\beta} \left( x \right) = \left(
\begin{matrix}
x_1 & 0 \\
0 & \delta x_1^2 \pm 2 x_1 x_2^2 + \beta x_2^4
\end{matrix}
\right)
\end{equation}
where
\begin{math}
\delta = \pm 1
\end{math}
and
\begin{math}
\beta \in \mathbb{R}
\end{math}
is a moduli parameter satisfying
\begin{math}
\beta \neq \delta
\end{math},
and
\begin{equation}
A_{2,1} \left( x \right) = \left(
\begin{matrix}
x_1 & 0 \\
0 & d_1 \left( x_1 + d_2 x_2^2 \right)^2 + x_2^5
\end{matrix}
\right),
\end{equation}
where
\begin{math}
d_i = \pm 1
\end{math}. They have the modality
\begin{math}
2
\end{math},
\begin{math}
\mathcal{G}_e
\end{math}-codimension
\begin{math}
10
\end{math}
and is adjacent to Class $\# 10$ in Table~\ref{table:2}.

\begin{math}
A_{\alpha,\beta}
\end{math}
and
\begin{math}
A_{\beta}
\end{math}
are non-simple but symmetrically quasi-homogeneous with weights
\begin{math}
\lambda_1 = 2
\end{math}
and
\begin{math}
\lambda_2 = 1
\end{math}.
\begin{math}
A_{2,1}
\end{math}
is not symmetrically quasi-homogeneous. This can be shown as follows.  The Lie algebra
\begin{math}
L \mathcal{D}_{A_{2,1}}
\end{math}
of the isotropy group
\begin{math}
\mathcal{D}_{A_{2,1}}
\end{math}
has the following form
\begin{multline}
L \mathcal{D}_{A_{2,1}} = \\
\langle \left( 16 d_1 x_1 x_2 + 25 x_1 x_2^2 \right) \frac{\partial}{\partial x_1} + \left( 2d_1 d_2 x_1 + 10 d_1 x_2^2 + 10 x_2^3 \right) \frac{\partial}{\partial x_2}, \\
\left( 5 x_1^2 + d_2 x_1 x_2^2 \right) \frac{\partial}{\partial x_1} + 2 x_1 x_2 \frac{\partial}{\partial x_2}
\rangle_{C^\infty \left( \mathbb{R}^r, \mathbb{R} \right)}
\end{multline}
modulo
\begin{math}
\langle x_1, x_2 \rangle^\infty
\end{math}. Note that both of the generators have nilpotent linear parts. Since the eigenvalues of a linearized vector field are invariant under the action of
\begin{math}
\mathcal{D}
\end{math},
\begin{math}
L \mathcal{D}_{A_{2,1}}
\end{math}
can not contain a generalized Euler vector field with a positive total weight. This proves that
\begin{math}
A_{2,1}
\end{math}
is not symmetrically quasi-homogeneous.

Note that the condition that a map-germ being symmetrically quasi-homogeneous is stronger than one that each component of the map-germ is quasi-homogeneous. One such example is
\begin{equation}
A_h = \left(
\begin{matrix}
x_1^3 & x_1^2 x_2 + x_2^3 \\
x_1^2 x_2 + x_2^3 & x_2^5
\end{matrix}
\right),
\end{equation}
whose components are homogeneous polynomial with respect to
\begin{math}
x_1
\end{math}
and
\begin{math}
x_2
\end{math}. Its Lie algebra
\begin{math}
L \mathcal{D}_{A_h}
\end{math}
of the isotropy group
\begin{math}
\mathcal{D}_{A_h}
\end{math}
has the following form
\begin{multline}
L \mathcal{D}_{A_h} = \langle \left( x_1^2 + 3 x_1 x_2^3 \right) \frac{\partial}{\partial x_1} + \left( 3 x_1 x_2 + x_2^4 \right) \frac{\partial}{\partial x_2}, \\
\left( x_1 x_2^2 + 9 x_1^2 x_2^3 \right) \frac{\partial}{\partial x_1} + \left( 8 x_1^2 x_2 + 3 x_2^3 + 3 x_1 x_2^4 \right) \frac{\partial}{\partial x_2}, \\
\left( 117 x_1^3 x_2 - 7 x_1^4 x_2^2 + 192 x_1^2 x_2^4 + 60 x_2^6 + 56 x_1 x_2^7 \right) \frac{\partial}{\partial x_1} \\
+ \left( 36 x_1^4 + 171 x_1^2 x_2^2 - 21 x_1^3 x_2^3 + 156 x_1 x_2^5 \right) \frac{\partial}{\partial x_2}, \\
\left( 45 x_1 x_2^4 - 95 x_1^4 x_2^3 + 114 x_1^2 x_2^5 + 120 x_2^7 \right) \frac{\partial}{\partial x_1} \\
+ \left( 135 x_2^5 - 69 x_1^3 x_2^4 + 222 x_1 x_2^6 + 112 x_1 x_2^8 \right) \frac{\partial}{\partial x_2} \rangle_{C^\infty \left( \mathbb{R}^r, \mathbb{R} \right)}
\end{multline}
modulo
\begin{math}
\langle x_1, x_2 \rangle^\infty
\end{math}. Note that both of the generators have nilpotent linear parts and thus
\begin{math}
A_h
\end{math}
is not symmetrically quasi-homogeneous.

For a non symmetrically quasi-homogeneous map-germ, we can get a unimodular classification as follows. To demonstrate it, let us consider
\begin{math}
A_{2,1}
\end{math}(\begin{math} d_1 = d_2 = 1 \end{math})
as an example. First, let us classify
\begin{math}
\mbox{Vol}
\end{math}
by the action of
\begin{math}
\mathcal{D}_{A_{2,1}}
\end{math}. If $\Omega$ is a volume form-germ and $V$ is a vector field-germ then $ d \left(  V \rfloor\Omega \right) = \left( \mathrm{div} V \right) \Omega$. Thus the set of infinitesimal actions of
\begin{math}
\mathcal{D}_{A_{2,1}}
\end{math}
to a volume form-germ
\begin{math}
\Omega
\end{math}
is
\begin{equation}
T \mathcal{D}_{A_{2,1}} \left( \Omega \right) = \left\{ \left( \mathrm{div} V \right) \Omega \left| V \in L \mathcal{D}_{A_{0,1}} \right. \right\}.
\end{equation}
This set is a vector subspace of
\begin{math}
\Lambda^2
\end{math}
over
\begin{math}
\mathbb{R}
\end{math}.
Its quotient vector space becomes
\begin{equation}
\cfrac{\Lambda^2}{T \mathcal{D}_{A_{2,1}} \left( \Omega \right)} = \langle 1 \rangle_{\mathbb{R}}
\end{equation}
modulo
\begin{math}
\langle x_1, x_2 \rangle^\infty
\end{math}. This can be shown as follows. By fixing the standard volume form
\begin{math}
dx_1 \wedge dx_2
\end{math}, we get the isomorphism
\begin{math}
\Lambda^r \cong C^\infty \left( \mathbb{R}^2, \mathbb{R} \right)
\end{math}
by identifying
\begin{math}
f \left( x \right) dx_1 \wedge dx_2
\end{math}
with
\begin{math}
f \left( x \right)
\end{math}. In what follows, we identify the two through the isomorphism. First note that
\begin{multline}
T \mathcal{D}_{A_{2,1}} \left( \Omega \right) =   \\
\left\{\left( 16 x_1 x_2 + 25 x_1 x_2^2 \right) \frac{\partial f_1}{\partial x_1} +  2\left( x_1 + 5 x_2^2 + 5 x_2^3 \right) \frac{\partial f_1}{\partial x_2} + \left( 36 x_2 + 55 x_2^2 \right) f_1,\right. \\
\left. \left( 5 x_1^2 + x_1 x_2^2 \right) \frac{\partial f_2}{\partial x_1} + \left( 2 x_1 x_2 \right) \frac{\partial f_2}{\partial x_2} + \left( 12 x_1 + x_2^2 \right) f_2 \left| f_1, f_2 \in C^\infty \left( \mathbb{R}^2, \mathbb{R} \right) \right. \right\} \label{eq:gen}
\end{multline}
modulo
\begin{math}
\langle x_1, x_2 \rangle^\infty
\end{math}. In what follows, we show
\begin{math}
\cfrac{\Lambda^2}{T \mathcal{D}_{A_{2,1}} \left( \Omega \right)} = \langle 1 \rangle_{\mathbb{R}}
\end{math}
modulo
\begin{math}
\langle x_1, x_2 \rangle^\ell
\end{math}
inductively for
\begin{math}
\ell \in \mathbb{N}
\end{math}. If we set
\begin{math}
f_1 = c_1
\end{math}
and
\begin{math}
f_2 = c_2
\end{math}
in Eq.~\eqref{eq:gen} where
\begin{math}
c_i \in \mathbb{R}
\end{math}, we get
\begin{math}
36 c_1 x_2, 12 c_2 x_1 \in T \mathcal{D}_{A_{2,1}} \left( \Omega \right)
\end{math}
modulo
\begin{math}
\langle x_1, x_2 \rangle^2
\end{math}.
This implies that
\begin{math}
\cfrac{\Lambda^2}{T \mathcal{D}_{A_{2,1}} \left( \Omega \right)} = \langle 1 \rangle_{\mathbb{R}}
\end{math}
modulo
\begin{math}
\langle x_1, x_2 \rangle^2
\end{math}. Let us assume
\begin{math}
\cfrac{\Lambda^2}{T \mathcal{D}_{A_{2,1}} \left( \Omega \right)} = \langle 1 \rangle_{\mathbb{R}}
\end{math}
modulo
\begin{math}
\langle x_1, x_2 \rangle^{\ell}
\end{math}
for
\begin{math}
\ell \ge 2
\end{math}
and we show this holds modulo
\begin{math}
\langle x_1, x_2 \rangle^{\ell+1}
\end{math}. To show that, note that
\begin{multline}
\left( x_2^2 \frac{\partial^2}{\partial x_2^2} + \frac{144}{7} x_1 x_2^2 \frac{\partial^2}{\partial x_1^2} - \frac{121}{14} x_2^3 \frac{\partial^2}{\partial x_1 \partial x_2} + \frac{10}{7} x_2^3 \frac{\partial^2}{\partial x_2^2} \right. \\
- \frac{225}{7} x_1 x_2^3 \frac{\partial^2}{\partial x_1^2} - \frac{87}{7} x_2^4 \frac{\partial^2}{\partial x_1 \partial x_2} - \frac{1}{7} x_2^4 \frac{\partial^2}{\partial x_2^2} + \frac{216}{35} x_2 \frac{\partial}{\partial x_2} \\
\left. - \frac{324}{7} x_2^2 \frac{\partial}{\partial x_1} + \frac{66}{7} x_2^2 \frac{\partial}{\partial x_2} - \frac{495}{7} x_2^3 \frac{\partial}{\partial x_1} \right) f_3 \in T \mathcal{D}_{A_{2,1}} \left( \Omega \right)
\end{multline}
holds for any
\begin{math}
f_3 \in T \mathcal{D}_{A_{2,1}} \left( \Omega \right)
\end{math}. By setting
\begin{math}
f_3 \left( x \right) = x_2^\ell / \left( \ell \left( \ell-1 \right) + \frac{216}{35} \ell \right)
\end{math}, we get
\begin{math}
x_2^\ell \in T \mathcal{D}_{A_{2,1}} \left( \Omega \right)
\end{math}
modulo
\begin{math}
\langle x_1, x_2 \rangle^{\ell + 1}
\end{math}. If
$f_2 \left( x \right) =$ $ \frac{1}{2 \left( \ell-j+1 \right)} x_1^{j-1} x_2^{\ell-j+1}$
(\begin{math}
1 \le j \le \ell
\end{math}) in the second generator in Eq.~\eqref{eq:gen}, then we get
\begin{math}
x_1^j x_2^{\ell-j} \in T \mathcal{D}_{A_{2,1}} \left( \Omega \right)
\end{math}
modulo
\begin{math}
\langle x_1, x_2 \rangle^{\ell + 1}
\end{math}. This proves the claim. 

This means that two volume form-germs
\begin{math}
\Omega_0
\end{math}
and 
\begin{math}
\Omega_1
\end{math}
are formally 
\begin{math}
L \mathcal{D}_{A_{2,1}}
\end{math}-equivalent if and only if 
\begin{math}
\Omega_0 \left( 0 \right) = \Omega_1 \left( 0 \right)
\end{math}
holds where 
\begin{math}
\Omega_0 \left( 0 \right)
\end{math}
is a value of 
\begin{math}
\Omega_0
\end{math}
at the origin. By Theorem~\ref{thm2.8}, if the two volume form-germs are 
\begin{math}
L \mathcal{D}_{A_{2,1}}
\end{math}-equivalent, they are
\begin{math}
\mathcal{D}_{A_{2,1}}
\end{math}-isotopic. Therefore, any 
\begin{math}
\Omega \in \mbox{Vol}
\end{math}
is formally 
\begin{math}
\mathcal{D}_{A_{2,1}}
\end{math}-isotopic to 
\begin{math}
\gamma^{-1} \; dx_1 \wedge dx_2
\end{math}
where 
\begin{math}
\gamma \neq 0
\end{math}
is a moduli parameter.

This means that formally
\begin{equation}
A_{2,1} \left( x \right) = \left(
\begin{matrix}
x_1 & 0 \\
0 & \left( \gamma x_1 + x_2^2 \right)^2 + x_2^5
\end{matrix}
\right),
\end{equation}
is a tri-modal map-germ relative to volume-preserving equivalence.

\section{Conclusion} \label{sec:conclusion}
We have introduced the volume-preserving equivalence among symmetric matrix-valued map-germs which is the unimodular version of Bruce's $\mathcal{G}$-equivalence. The key concept to deduce unimodular classification out of classification relative to
\begin{math}
\mathcal{G}
\end{math}-equivalence is symmetrical quasi-homogeneity. If a $\mathcal{G}$-equivalence class contains a symmetrically quasi-homogeneous representative, the class coincides with that relative to the volume-preserving equivalence (up to orientation reversing diffeomorphism in case if the ground field is real). By using that we have shown that all the simple classes relative to
\begin{math}
\mathcal{G}
\end{math}-equivalence in Bruce's list coincides with those relative to the volume preserving equivalence. Then, we have classified symmetric matrix-valued map-germs
\begin{math}
\left( \mathbb{R}^r, 0 \right) \rightarrow \left( \textnormal{Sym}_n, O_n \right)
\end{math}
of corank at most $1$ for the cases $r = 2, n = 2, 3$ and of
\begin{math}
\mathcal{G}_e
\end{math}-codimension less than $9$ and we have shown some of the normal forms split into two different unimodular singularities. We have provided several examples to illustrate that non simplicity does not imply non symmetrical quasi-homogeneity and the condition that a map-germ is symmetrically quasi-homogeneous is stronger than one that each component of the map-germ is quasi-homogeneous. We have also presented an example of non symmetrically quasi-homogeneous normal form relative to
\begin{math}
\mathcal{G}
\end{math}
and its corresponding formal unimodular normal form.



\begin{thebibliography}{9}
\bibitem{A}  V. I. Arnold, S. M. Gusein-Zade, A. N. Varchenko, \emph{
Singularities of Differentiable Maps}, Vol. 1, Birhauser, Boston, 1985.

\bibitem{Braam} P. Braam and H. Duistermaat, \emph{Normal forms of real symmetric systems with multiplicity}, Indag. Math., N. S., vol. 4, number 4, 1993, 407--421.

\bibitem{B} J. W. Bruce, \emph{ On families of symmetric matrices}, Moscow Mathematical Journal, vol. 3, number 2, April-June 2003, 335--360.

\bibitem{Bruce2002} J. W. Bruce, V. V. Goryunov and V. M. Zakalyukin, \emph{Sectional Singularities and Geometry of Families of Planar Quadratic Forms}, Trends in singularities, Trends Math., Birkh{\"a}user, Basel, 2002, 83--97.

\bibitem{Bruce1997}
J. W. Bruce and N. P. Kirk and A. A. du Plessis, \emph{Complete transversals and the classification of singularities}, Nonlinearity, vol. 10, number 1, 1997, 253--275.

\bibitem{CV_I} Y. Colin de Verdi\`ere, \emph{The level crossing problem in semi-classical analysis, I. The symmetric case}, Ann. Inst. Fourier, Grenoble, vol. 53, number 4, 2003, 1023--1054.

\bibitem{Damon}
J. Damon,
\emph{The unfolding and determinacy theorems for subgroups of $\mathcal{A}$ and
$\mathcal{K}$}.
 Memoirs of A.M.S. vol. 50,  No. 306, 1984.
 
\bibitem{D} W. Domitrz, \emph{Zero-dimensional symplectic isolated complete intersection singularities}, J. Singul. 6 (2012), 19-26.

\bibitem{DJZ1} W. Domitrz, S. Janeczko, M. Zhitomirskii, \emph{ Relative Poincare lemma, contractibility, quasi-homogeneity and vector fields tangent to a singular variety}, Illinois Journal of Mathematics, vol. 48, number 3, Fall 2004, 803--835.

\bibitem{DJZ2} W. Domitrz, S. Janeczko, M. Zhitomirskii, {\em Symplectic singularities of varietes: the method of algebraic restrictions}, Journal f\"ur die reine und angewandte Mathematik 618(2008), 197-235.

\bibitem{DR} W. Domitrz, J. H. Rieger, \emph{Volume preserving subgroups of
$\mathcal A$ and $\mathcal K$ and singularities in unimodular
geometry}, Mathematische Annalen
   345(2009), 783-817.



\bibitem{Honda} N. Honda, T. Kawai, Y. Takei, \emph{Virtual Turning Points}, SpringerBriefs in Mathematical Physics, 2015.














\bibitem{Lancaster2005} P. Lancaster and L. Rodman, \emph{Canonical Forms for Hermitian Matrix Pairs under Strict Equivalence and Congruence}, SIAM Review, vol. 47, number 3, September 2005, 407--443.

\bibitem{Lando} S. K. Lando,  \emph{Normal forms of the degrees of a volume form},  Funct Anal Its Appl 19 (1985), 146–148. 



\bibitem{Martinet}
J. Martinet,
\emph{Singularities of Smooth Functions and Maps},
 London Math. Soc. Lecture Note Series 58, Cambridge Univ. Press, 1982.

\bibitem{Mather1968}
J. Mather,
\emph{Stability of $C^\infty$-Mappings III. Finitely Determined Map Germs}.
 Publ. Math. I.H.E.S.
vol. 36, 1968, 127--156.


\bibitem{Mather1969}
J. Mather,  \emph{Stability of C$^\infty$ mappings, IV: Classification of stable germs by $\mathbb{R}$ algebras}, Publ. Math. I. H. E. S., vol. 37, 1969, 223--248.

\bibitem{Saito1971} K. Saito, \emph{Quasihomogene isolierte Singularit{\"a}ten von Hyperfl{\"a}chen}, Inventiones mathematicae 14, 1971, 123-142.

\bibitem{Varchenko}  A. N. Varchenko, \emph{Local classification of volume forms in the presence of a hypersurface}, Funct Anal Its Appl 19 (1985), 269–276.

\bibitem{Wall1981} C. T. C. Wall, \emph{Finite Determinacy of Smooth Map-Germs}, Bull. London Math. Soc., vol. 13, number 6, November 1981, 481--539.

















\end{thebibliography}
\end{document}